\documentclass[12pt]{amsproc}

\usepackage{geometry, enumerate, amsmath, amssymb, amsthm, amscd, mathrsfs, hyperref, graphicx, color, enumerate, float, verbatim, colonequals, url}

\numberwithin{equation}{section}
\newtheorem{thm}[equation]{Theorem}
\newtheorem{cor}[equation]{Corollary}

\newtheorem{lem}[equation]{Lemma}

\newtheorem{prop}[equation]{Proposition}

\newenvironment{myproof}[1][\proofname]{%
  \proof[\scshape Proof #1]%
}{\endproof}

\theoremstyle{definition}\newtheorem{Rem}[equation]{Remark}
\theoremstyle{definition}\newtheorem{Def}[equation]{Definition}
\theoremstyle{definition}
\theoremstyle{definition}\newtheorem{Not}[equation]{Notation}

\newcommand{\N}{\mathbb{N}}

\newcommand{\F}{\mathbb{F}}

\renewcommand{\le}{\leqslant}
\renewcommand{\ge}{\geqslant}

\newcommand{\mcA}{\mathcal{A}}

\newcommand{\mcC}{\mathcal{C}}

\newcommand{\mcI}{\mathcal{I}}

\newcommand{\mcU}{\mathcal{U}}

\newcommand{\msJ}{\mathscr{J}}

\newcommand{\msS}{\mathscr{S}}
\newcommand{\msT}{\mathscr{T}}

\newcommand{\lcm}{\text{lcm}}

\newcommand{\olM}{\overline{M}}

\newcommand{\olS}{\overline{S}}

\newcommand{\whI}{\widehat{I}}
\newcommand{\whJ}{\widehat{J}}
\newcommand{\whS}{\widehat{S}}
\newcommand{\whT}{\widehat{T}}
\newcommand{\whX}{\widehat{X}}

\newcommand*\xbar[1]{%
   \hbox{%
     \vbox{%
       \hrule height 0.5pt 
       \kern0.5ex
       \hbox{%
         \kern-0.1em
         \ensuremath{#1}%
         \kern-0.1em
       }%
     }%
   }%
}

\title{The covering numbers of rings}

\date{\today}

\author{Eric Swartz}
\address{Department of Mathematics, William \& Mary, P.O. Box 8795, Williamsburg, VA 23187-8795, USA}
\email{easwartz@wm.edu}

\author{Nicholas J. Werner}
\address{Department of Mathematics, Computer and Information Science, State University of New York College at Old Westbury, Old Westbury, NY 11560, USA}
\email{wernern@oldwestbury.edu}

\begin{document}

\begin{abstract}
A cover of an associative (not necessarily commutative nor unital) ring $R$ is a collection of proper subrings of $R$ whose set-theoretic union equals $R$. If such a cover exists, then the covering number $\sigma(R)$ of $R$ is the cardinality of a minimal cover, and a ring $R$ is called $\sigma$-elementary if $\sigma(R) < \sigma(R/I)$ for every nonzero two-sided ideal $I$ of $R$.  If $R$ is a ring with unity, then we define the unital covering number $\sigma_u(R)$ to be the size of a minimal cover of $R$ by subrings that contain $1_R$ (if such a cover exists), and $R$ is $\sigma_u$-elementary if $\sigma_u(R) < \sigma_u(R/I)$ for every nonzero two-sided ideal of $R$. In this paper, we classify all $\sigma$-elementary unital rings and determine their covering numbers.  Building on this classification, we are further able to classify all $\sigma_u$-elementary rings and prove $\sigma_u(R) = \sigma(R)$ for every $\sigma_u$-elementary ring $R$.  We also prove that, if $R$ is a ring without unity with a finite cover, then there exists a unital ring $R'$ such that $\sigma(R) = \sigma_u(R')$, which in turn provides a complete list of all integers that are the covering number of a ring.  Moreover, if \[\mathscr{E}(N) := \{m : m \le N, \sigma(R) = m \text{ for some ring } R\},\]
then we show that $|\mathscr{E}(N)| = \Theta(N/\log(N))$, which proves that almost all integers are not covering numbers of a ring. 
\end{abstract}

\maketitle

\section{Introduction}\label{sect:intro}

In this article, rings are assumed to be associative, but need not be commutative nor have a multiplicative identity. Our goal is to study the ways in which a ring $R$ can be expressed as a union of proper subrings. Here, we use subring in the weakest sense: $S \subseteq R$ is a subring when $S$ is an additive group and is closed under multiplication. In particular, $S$ need not contain a multiplicative identity, even if $R$ itself is unital. A \textit{cover} of $R$ is a collection $\mcC$ of proper subrings of $R$ such that $R = \bigcup_{S \in \mcC} S$. When such a cover exists, we say that $R$ is \textit{coverable}, and define the \textit{covering number} $\sigma(R)$ to be the cardinality of a minimal cover.  If $R$ is not coverable, we set  $\sigma(R) = \infty$.

The origin of this problem comes from group theory, where there is an analogous notion of a cover (by subgroups) for a noncyclic group, and covering problems have a long and fascinating history.  That no group is the union of exactly two proper subgroups is an easy exercise (and, in fact, a Putnam Competition problem from 1969; see \cite[Chapter 3, Exercise 15]{Gallian}). The earliest known consideration of covering numbers of groups was perhaps by Scorza \cite{Scorza}, who proved that a group is the union of three proper subgroups if and only if it has a quotient isomorphic to the Klein $4$-group; a nice proof of this result, which has been rediscovered a number of times through the years, can be found in \cite{Bhargava}.  

Problems related to covering numbers of groups have received considerable attention in recent years.  Cohn \cite{Cohn} proved that there exists a group with covering number $p^n + 1$ for every prime $p$ and positive integer $n$ and conjectured that every solvable group has a covering number of this form.  This conjecture was proven by Tomkinson \cite{Tomkinson}, who also proved that there is no group with covering number equal to 7.  Covering numbers have been determined for various families of simple and almost simple groups (see, for instance, \cite{Britnell1, Britnell2, BryceFedriSerena, Holmes, Maroti, Swartz}). As noted above, no group has covering number 7, and other natural numbers exist that are not the covering number of a group (the next smallest examples being 11 \cite{DetomiLucchini} and 19 \cite{Garonzi}). All the integers $N \le 129$ that are not the covering number of a group were determined in \cite{GaronziKappeSwartz}, which also contains a good summary of the history and recent developments of the related work for groups. It is not known whether there are infinitely many positive integers that are not the covering number of a group.

A natural question then arises: what can be said about covering numbers of other algebraic structures?  The covering number of a vector space by proper subspaces was determined by Khare \cite{Khare} (see also \cite{Clark}).  The 2014 article by Kappe \cite{Kappe} provides a survey of recent research, but there have been some notable results since its publication.  Gagola III and Kappe \cite{GagolaKappe} proved that for every integer $n > 2$, there exists a loop whose covering number is exactly $n$, and Donoven and Kappe \cite{DonovenKappe} proved that the covering number of a finite semigroup that is not a group and not generated by a single element is always two. Furthermore, they showed that for each $n \ge 2$, there exists an inverse semigroup whose covering number (by inverse subsemigroups) is exactly $n$.  Recently, the covering number of modules by proper submodules was studied in \cite{KhareTikaradze} and \cite{Ghosh}.

The related question for rings has also received significant attention, especially in the last decade.  Lucchini and Mar\'oti \cite{LucchiniMaroti} determined all rings with covering number equal to three.  However, as noted by Kappe \cite[p. 86]{Kappe}, ``the solution is less simple than the group case,'' and a ring has covering number three if and only if it has a homomorphic image isomorphic to one of five rings.  In light of this result, Kappe \cite[p. 87]{Kappe} further notes that a characterization of rings which are the union of exactly four proper subrings ``does not seem to be a feasible problem for rings.''  The difficulty of this problem is further illuminated in recent work by Cohen \cite{Cohen}, in which a strategy is presented to classify unital rings with a given covering number and a partial classification is given of unital rings whose covering number is four.  Moreover, for each integer $n$, $3 \le n \le 12$, examples are known of finite rings (with unity) that have covering number $n$, and no integers have thus far been ruled out from being the covering number of a ring.  Indeed, Kappe \cite[p. 87]{Kappe} further writes, ``An interesting question would be if there are integers $n > 2$ that are not the covering number of a ring.''  Other recent works on this problem include \cite{CaiWerner, Crestani, PeruginelliWerner, SwartzWerner, Werner}.   

The goal of this paper is to provide a general method to calculate $\sigma(R)$ for any ring $R$ with a finite covering number, and consequently to determine all integers that occur as the covering number of a ring.  Our main result (Theorem \ref{thm:main}) allows one to describe the covering numbers of all unital rings. A variation on this theorem (Theorem \ref{thm:rngtoring}, Theorem \ref{thm:mainunital}) can be used to compute the covering number of any nonunital ring. Thus, we are able to characterize all the positive integers that are the covering number of a ring, and prove that the set of such integers has density 0 (Corollary \ref{cor:density}). 

Our approach to these problems is based on the following easy observation: if $R$ has a two-sided ideal $I$ such that $R/I$ admits a cover $\mcC$, then we may form a cover of $R$ by taking the inverse images of the subrings in $\mcC$ under the natural homomorphism $R \to R/I$. Thus, $\sigma(R) \le \sigma(R/I)$.  Of capital importance are those rings for which the inequality $\sigma(R) \le \sigma(R/I)$ is strict for all $I \ne \{0\}$. 

\begin{Def}\label{def:sigmaelementary}
A ring $R$ is said to be \textit{$\sigma$-elementary} if $\sigma(R) < \sigma(R/I)$ for every nonzero two-sided ideal $I$ of $R$. Note that a $\sigma$-elementary ring $R$ must be coverable, since $\sigma(R) < \sigma(\{0\}) = \infty$.
\end{Def}

Observe that if $\sigma(R)$ is finite, then either $R$ is $\sigma$-elementary, or $R$ has a proper, $\sigma$-elementary residue ring with the same covering number. Moreover, if an integer $n$ occurs as the covering number of a ring, then there exists a $\sigma$-elementary ring $R$ such that $\sigma(R)=n$. Thus, many questions about covering numbers of rings can be reduced to the case of $\sigma$-elementary rings. 

In a recent paper \cite{SwartzWernerI}, the authors introduced a new family of $\sigma$-elementary rings, called rings of AGL-type (see Definition \ref{def:AGL}) and determined their covering numbers.  In Theorem \ref{thm:main} below, we will prove that rings of AGL-type, along with three other infinite families of rings (which have been studied previously), provide a complete list of all $\sigma$-elementary rings with unity. Furthermore, we give formulas for the covering numbers of each family of $\sigma$-elementary rings.

In order to state Theorem \ref{thm:main}, we must introduce some notation. For a prime power $q$, $\F_q$ is the finite field of order $q$, and $M_n(q)$ is the $n \times n$ matrix ring with entries from $\F_q$. The Jacobson radical of a ring $R$ is denoted by $\msJ(R)$, and when $R$ contains unity, the unit group of $R$ is $R^\times$. When $R$ is a commutative ring and $M$ is an $R$-module, the idealization of $R$ and $M$, denoted $R (+) M$, is the ring \[R(+)M:=\{(r,m) \mid r \in R, m \in M\},\] with multiplication $(r_1, m_1)\cdot (r_2, m_2) = (r_1 r_2, r_1 m_2 + r_2 m_1)$. Of particular importance is the idealization $\F_q (+) \F_q^2$ of $\F_q$ with the 2-dimensional vector space $\F_q^2$. This ring can be represented as the following ring of $3 \times 3$ matrices:
\begin{equation*}
\F_q (+) \F_q^2 \cong \left\{ \begin{pmatrix} a & b & c \\ 0 & a & 0 \\ 0 & 0 & a \end{pmatrix} : a, b, c \in \F_q \right\}.
\end{equation*}

Next, we give some definitions from \cite{SwartzWernerI} related to rings of AGL-type.

\begin{Def}\label{def:AGL}
Given two powers $q_1$ and $q_2$ of a prime $p$, we define $q_1 \otimes q_2$ to be the order of the field compositum of $\F_{q_1}$ and $\F_{q_2}$, which is $\F_{q_1} \otimes_{\F_p} \F_{q_2}$. Observe that when $q_1 = p^{d_1}$ and $q_2 = p^{d_2}$, we have $q_1 \otimes q_2 = p^{\lcm(d_1, d_2)}$.  

Let $n \ge 1$ and let $q = q_1 \otimes q_2$. We define $A(n,q_1,q_2)$ to be the following subring of $M_{n+1}(q)$:
\begin{equation*}
A(n,q_1,q_2) := \left(\text{\begin{tabular}{c|c} $M_n(q_1)$ & $M_{n \times 1}(q)$ \\ \hline $0$ & $\F_{q_2}$ \end{tabular} }\right),
\end{equation*}
and call this a ring of \textit{AGL-type}. The construction and terminology for $A(n,q_1,q_2)$ were inspired by representations of the affine general linear group ${\rm AGL}(n,q)$, which is isomorphic to a subgroup of $A(n,q,q)^\times$.
\end{Def}


The formulas to compute $\sigma(R)$ require some accessory functions. For a prime power $q = p^d$, we take ${\rm Irr}(p,d)$ to be the set of all monic irreducible polynomials in $\F_p[x]$ of degree $d$.  We define 
\[ \tau(q) := \begin{cases}
               p, &\text{ if } d = 1,\\
               |{\rm Irr}(p,d)| + 1, &\text{ if } d > 1.
              \end{cases}
\]
We let $\omega$ denote the prime omega function, which counts the number of distinct prime divisors of a natural number. Note that $\omega(1) = 0$. For $q=p^d$, we define
\[ \nu(q) := \begin{cases}
              1, &\text{ if } d = 1,\\
              \omega(d), &\text{ if } d > 1,
             \end{cases}\]
which counts the number of maximal subrings of the field $\F_q$. The variation in the definitions between $\omega$ and $\nu$ is due to the fact that when $q=p$, $\{0\}$ is the only maximal subring of $\F_q$; but, when $q > p$, the maximal subrings of $\F_q$ are exactly the maximal subfields of prime index. Finally, we let $\binom{n}{k}_q$ denote the $q$-binomial coefficient, which counts the number of $k$-dimensional subspaces of $\F_q^n$:
\begin{equation*}
\binom{n}{k}_q = \frac{(q^n - 1)(q^{n-1}-1) \cdots (q^{n - (k-1)}-1)}{(q^k -1)(q^{k-1} -1)\cdots (q-1)}.
\end{equation*}

We can now state our first main result.

\begin{thm}\label{thm:main}
Let $R$ be a $\sigma$-elementary ring with unity.  Then, one of the following holds.
\begin{enumerate}
\item If $R$ is commutative and semisimple, then for some prime power $q=p^d$, $R \cong \bigoplus_{i=1}^{\tau(q)} \F_q$ and \[\sigma(R) = \tau(q) \nu(q) + d \binom{\tau(q)}{2}.\]
\item If $R$ is commutative but not semisimple, then for some prime power $q$, $R \cong \F_q (+) \F_q^2$ and \[\sigma(R) = q + 1.\]
\item If $R$ is noncommutative and semisimple, then for some prime power $q$ and integer $n \ge 2$, $R \cong M_n(q)$ and
 \[ \sigma(R) = \frac{1}{a} \prod_{k=1,\\ a \nmid k}^{n-1} (q^n - q^k) + \sum_{k=1,\\ a \nmid k}^{\lfloor n/2 \rfloor} {n \choose k}_q,\]
where $a$ is the smallest prime divisor of $n$.
\item If $R$ is noncommutative and not semisimple, then $R \cong A(n, q_1, q_2)$ and one of the two cases below holds. Let $q= q_1 \otimes q_2 = q_1^d$, and if $n \ge 2$, then let $a$ be the smallest prime divisor of $n$.
\begin{enumerate}[(i)]
\item $n=1$ and $(q_1, q_2) \ne (2, 2)$ or $(4,4)$. In this case, $\sigma(R) = q+1$.
\item $n \ge 3$, $d < n - (n/a)$, and $(n, q_1) \ne (3, 2)$. In this case,
\begin{equation*}
\sigma(R) = q^n + \binom{n}{d}_{q_1} + \omega(d).
\end{equation*}
\end{enumerate}
\end{enumerate}
\end{thm}

Portions of Theorem \ref{thm:main} have been proved in earlier papers. The classification of commutative unital $\sigma$-elementary rings was done in \cite{SwartzWerner}, and the covering numbers of the rings in parts (1) and (2) were determined in \cite{Werner}. The formula for the covering number of $M_n(q)$ is due to Crestani, Lucchini, and Mar\'{o}ti \cite{Crestani,LucchiniMaroti}. Since $M_n(q)$ is simple, it is clearly $\sigma$-elementary, and the fact that these are the only noncommutative semisimple $\sigma$-elementary rings follows from the classification of covering numbers for finite semisimple rings done in \cite{PeruginelliWerner}. The determination of when $A(n, q_1, q_2)$ is $\sigma$-elementary, and, in that event, the computation of its covering number, were completed in \cite{SwartzWernerI}.  The content of the present paper is a proof that any $\sigma$-elementary unital ring falls into one of the four classes listed in Theorem \ref{thm:main}.

While Theorem \ref{thm:main} handles the situation where $R$ contains unity, it leaves open the problem of determining which integers are covering numbers of nonunital rings. This more general case can be reduced to the unital case by considering covers of $R$ in which every subring contains $1_R$.

\begin{Def}\label{def:sigmau}
Let $R$ be a ring with unity. If $R$ can be covered by unital subrings (i.e.\ those containing $1_R$), then the \textit{unital covering number} $\sigma_u(R)$ is defined to be the cardinality of a minimal cover by unital subrings. We say $R$ is \textit{$\sigma_u$-elementary} if $\sigma_u(R) < \sigma_u(R/I)$ for every nonzero two-sided ideal $I$ of $R$. Note that a $\sigma_u$-elementary ring necessarily contains unity, and admits a cover by unital subrings.
\end{Def}

\begin{thm}\label{thm:rngtoring}
Let $R$ be a ring without unity that has a finite cover.  Then, there exists a ring $R'$ with unity such that $\sigma(R) = \sigma_u(R')$.  Thus, every covering number of a nonunital ring by subrings occurs as the covering number of a unital ring by unital subrings.
\end{thm}

Via Theorem \ref{thm:rngtoring}, we can describe the covering numbers of all rings (with or without unity) by classifying $\sigma_u$-elementary rings. Our classification is very similar to that of unital $\sigma$-elementary rings, which was given in Theorem \ref{thm:main}.

\begin{thm}\label{thm:mainunital}
Let $R$ be a $\sigma_u$-elementary ring. 
\begin{itemize}
    \item[(1)]If $R$ is $\sigma$-elementary, then $R$ is one of the rings listed in Theorem \ref{thm:main}, with the exception that  $R \not\cong \bigoplus_{i=1}^p \F_p$. Moreover, any ring $R \not\cong \bigoplus_{i=1}^p \F_p$ listed in Theorem \ref{thm:main} is $\sigma_u$-elementary. 
    \item[(2)]If $R$ is not $\sigma$-elementary, then either $R \cong \bigoplus_{i=1}^{p+1} \F_p$ with $\sigma_u(R) = p+ \binom{p}{2}$, or $R \cong A(1,2,2)$ with $\sigma_u(R)=3$.  
\end{itemize}    
In particular, if $R$ is $\sigma_u$-elementary, then $\sigma_u(R) = \sigma(R)$, and the integers that are covering numbers of rings with unity are the same as the integers that are unital covering numbers of rings with unity.
\end{thm}

We remark that Theorem \ref{thm:mainunital} is likely of independent interest: for many, the definition of a ring includes the existence of a multiplicative identity, and subrings likewise are required to contain this unital element; see, e.g., \cite[Chapter II]{Lang}.  (Those who adopt this convention often refer to a ring without unity as an \textit{rng}.)  In this setting, the only covers would indeed be covers by unital subrings.    

In light of these results, any integer that is the covering number of a ring (with or without unity) or unital covering number of a ring with unity can be computed by the formulas in Theorem \ref{thm:main}.

\begin{cor}\label{cor:unitalsuffices}
Any integer that is the covering number or unital covering number of a ring occurs as the covering number of a ring with unity.  
\end{cor}

This, in turn, allows us to answer Question 1.2 from \cite{SwartzWerner}:

\begin{cor}\label{cor:not13}
 There does not exist an associative ring with covering number $13$.
\end{cor}

Even more, Theorems \ref{thm:main}, \ref{thm:rngtoring}, and \ref{thm:mainunital} allow us to prove that there are infinitely many integers that are not covering numbers (or unital covering numbers) of rings. In fact, almost all integers are not covering numbers of rings.

\begin{cor}\label{cor:density}
Let $\mathscr{E}(N) := \{m : m \le N, \sigma(R) = m \text{ for some ring } R\}.$
Then, for all $N \ge 5$, \[\frac{N}{\log N} < |\mathscr{E}(N)| < \frac{144 N}{\log N},\]
where $\log N$ denotes the binary (base $2$) logarithm of $N$.  In particular, we have $|\mathscr{E}(N)| = \Theta(N/\log(N))$ and
\[ \lim_{N \to \infty} \frac{|\mathscr{E}(N)|}{N} = 0.\]
\end{cor}

This paper is organized as follows.  For the majority of the article, we will focus on rings with unity. In Section \ref{sect:prelim}, we collect a number of background results that will be used in the rest of the article.  Section \ref{sect:Peirce} lays the groundwork for the proof of Theorem \ref{thm:main} by studying the Peirce decomposition of the Jacobson radical of a $\sigma$-elementary ring and providing numerical bounds of covering numbers in certain cases.  Section \ref{sect:J} is dedicated to proving that most terms in the Peirce decomposition of the Jacobson radical of a $\sigma$-elementary ring must in fact be $\{0\}$ (so, a $\sigma$-elementary ring cannot be ``too far'' from being semisimple), and Section \ref{sect:summands} in turn places severe restrictions on the direct summands of the semisimple residue of a $\sigma$-elementary ring.  This portion of the work culminates in Section \ref{sect:mainproof}, where we use the results of the previous sections to prove Theorem \ref{thm:main}.  

In Section \ref{sect:othermain}, we deal with the general case where rings need not contain a multiplicative identity. After proving Theorem \ref{thm:rngtoring}, we concentrate on proving Theorem \ref{thm:mainunital}, which is accomplished using the methods developed in earlier sections. Finally, in a short post-script (Section \ref{sect:E(N)bounds}), we employ the formulas for $\sigma(R)$ in Theorem \ref{thm:main} to prove Corollary \ref{cor:density}.   


\section{Frequently used results}\label{sect:prelim}
In this section, we summarize a number of results that will be referenced frequently later in the paper. The first lemma is elementary, and will be taken for granted throughout the article.

\begin{lem}\label{lem:basics}
Let $R$ be a ring with unity.
\begin{enumerate}[(1)]
\item $R$ is coverable if and only if $R$ cannot be generated (as a ring) by a single element.
\item If $R$ is noncommutative, then $R$ is coverable.
\item For any two-sided ideal $I$ of $R$, a cover of $R/I$ can be lifted to a cover of $R$. Hence, $\sigma(R) \le \sigma(R/I)$.
\item If each proper subring of $R$ is contained in a maximal subring, then we may assume that any minimal cover of $R$ consists of maximal subrings.
\end{enumerate}
\end{lem}


By Lemma \ref{lem:basics}(4), maximal subrings of $R$ are important in forming minimal covers. The next lemma provides conditions under which a maximal subring is guaranteed to be part of every cover of a ring.

\begin{lem}\label{lem:SigmaElementary}
Let $R$ be a coverable ring, and let $\mcC$ be a minimal cover of $R$. 
\begin{enumerate}[(1)]
\item \cite[Lemma 2.1]{Werner} If $M$ is a maximal subring of $R$ and $M \notin \mcC$, then $\sigma(M) \leq \sigma(R)$.

\item \cite[Lemma 2.2]{SwartzWerner} Let $S$ be a proper subring of $R$ such that $R=S \oplus I$ for some two-sided ideal $I$ of $R$. If $\sigma(R) < \sigma(R/I)$, then $S \subseteq T$ for some $T \in \mcC$. If, in addition, $S$ is a maximal subring of $R$, then $S \in \mcC$.

\end{enumerate}
\end{lem}

As shown in \cite[Theorem 2.2]{Werner}, when $R$ is a direct sum of rings, $\sigma(R)$ can be determined from the covering numbers of the summands of $R$ as long as all maximal subrings of $R$ respect the direct sum decomposition. While not explicitly stated in \cite{Werner}, the analogous result for unital maximal subrings and unital coverings of $R$ is also true, and can be proved in the same manner as \cite[Theorem 2.2]{Werner}.

\begin{lem}\label{lem:AMM2.2}\cite[Theorem 2.2]{Werner}
Let $R = \bigoplus_{i=1}^t R_i$ for rings $R_1, \ldots, R_t$. Assume that each maximal subring $M$ of $R$ has the form
\begin{equation*}
M = R_1 \oplus \cdots \oplus R_{i-1} \oplus M_i \oplus R_{i+1} \oplus \cdots \oplus R_t
\end{equation*}
for some $1 \le i \le t$. Then, $\sigma(R) = \min_{1 \le i \le t}\{\sigma(R_i)\}$.  Furthermore, if each $R_i$ is a unital ring and each unital maximal subring has the above form (with $M_i$ itself unital), then $\sigma_u(R) = \min_{1 \le i \le t}\{\sigma_u(R_i)\}$.
\end{lem}


From results of Neumann \cite[Lemma 4.1, 4.4]{Neumann} and Lewin \cite[Lemma 1]{Lewin}, it is known that if $R$ has a finite covering number, then there exists a two-sided ideal $I$ of $R$ of finite index such that $\sigma(R)=\sigma(R/I)$. Hence, the calculation of covering numbers reduces to the case of finite rings. In fact, a much stronger reduction is possible. Recall that $\msJ(R)$ denotes the Jacobson radical of a ring $R$.

\begin{lem}\label{lem:FullReduction}\cite[Theorem 3.12]{SwartzWerner}
Let $R$ be a (unital, associative) ring such that $\sigma(R)$ is finite. Then, there exists a two-sided ideal $I$ of $R$ such that $R/I$ is finite; $R/I$ has characteristic $p$; $\msJ(R/I)^2 = \{0\}$; and $\sigma(R/I) = \sigma(R)$.
\end{lem}


Thus, for unital rings with finite covers, the computation of $\sigma(R)$ reduces to the case where $R$ is finite of characteristic $p$ and $\msJ(R)^2=\{0\}$. When $R$ is finite, $R/\msJ(R)$ is semisimple, i.e., it is isomorphic to a direct sum of matrix rings over finite fields. The Wedderburn-Malcev Theorem \cite[Sec.\ 11.6, Cor.\ p.\ 211]{Pierce}, \cite[Thm.\ VIII.28]{McDonald} (sometimes called the Wedderburn Principal Theorem) provides a more detailed description of the structure of $R$, and how it relates to $R/\msJ(R)$.

\begin{thm}\label{thm:Wedderburn} (Wedderburn-Malcev Theorem)
Let $R$ be a finite ring with unity of characteristic $p$. Then, there exists an $\F_p$-subalgebra $S$ of $R$ such that $R = S \oplus \msJ(R)$, and $S \cong R/\msJ(R)$ as $\F_p$-algebras. Moreover, $S$ is unique up to conjugation by elements of $1+\msJ(R)$.
\end{thm}



Let $\msS(R)$ be the set of all semisimple complements to $\msJ(R)$ in $R$. 
By Theorem \ref{thm:Wedderburn}, all such complements are conjugate. So,  for any $S \in \msS(R)$, we have $R = S \oplus \msJ(R)$, and $S \cong R/\msJ(R)$. 

The decomposition given by Theorem \ref{thm:Wedderburn} allows us to characterize all of the maximal subrings of $R$.

\begin{lem}\label{lem:MaxSubringClassification}\cite[Theorem 3.10]{SwartzWerner}
Let $R$ be a finite ring with unity of characteristic $p$, let $M$ be a maximal subring of $R$, and let $J = \msJ(R)$.
\begin{enumerate}[(1)]
\item $J \subseteq M$ if and only if $M$ is the inverse image of a maximal subring of $R/J$.
\item $J \not\subseteq M$ if and only if $M = S \oplus \msJ(M)$, where $S \in \msS(R)$ and $\msJ(M) = M \cap J$ is an ideal of $R$ that is maximal among the subideals of $R$ contained in $J$.
\end{enumerate}
\end{lem}

\section{Peirce decomposition and bounds}\label{sect:Peirce}

We now turn our attention to proving that any $\sigma$-elementary unital ring falls into one of the four classes given in Theorem \ref{thm:main}. 
Throughout this section, we will assume that $R$ is a $\sigma$-elementary unital ring with $\sigma(R)$ finite. By Lemma \ref{lem:FullReduction}, we can assume that $R$ is finite of characteristic $p$ and has Jacobson radical $J = \msJ(R)$ with $J^2 = \{0\}$. Moreover, using Theorem \ref{thm:Wedderburn}, we have $R = S \oplus J$, where $S \cong R/J$ is a semisimple complement to $J$ in $R$. The set of all such complements is denoted by $\msS(R)$. For the remainder of this section, we fix $S \in \msS(R)$. Then, for some $N \ge 1$, $S = \bigoplus_{i = 1}^N S_i$, where each $S_i$ is a simple ring. For each $1 \le i \le N$, let $S_i = M_{n_i}(q_i)$, where $n_i \ge 1$ and $q_i$ is a power of $p$. Finally, let $e_i = 1_{S_i}$, so that $e_1, \ldots, e_N$ are orthogonal idempotents such that $\sum_{i = 1}^N e_i$ is the unity of $R$. 

Apply a two-sided Peirce decomposition to $J$ using the idempotents $e_1, \ldots, e_N$. This gives $J = \bigoplus_{1 \le i, j \le N} e_i J e_j$, where each $e_iJe_j$ is an ($S_i$, $S_j$)-bimodule. If $S_i=M_{n_i}(F_i)$ and $S_j=M_{n_j}(F_j)$, then ($S_i$, $S_j$)-bimodules are equivalent to modules over $M_{n_i n_j}(F)$, where $F = F_i \otimes_{\F_p} F_j$ is the compositum of $F_i$ and $F_j$. In particular, we may speak of the length of $e_i J e_j$ as an ($S_i$, $S_j$)-bimodule, i.e., the number of simple ($S_i$, $S_j$)-bimodules in a direct sum decomposition of $e_i J e_j$. Note that any simple ($S_i$, $S_j$)-bimodule has order $|F|^{n_in_j}$. 

\begin{Def}\label{def:lambda}
For all $1 \le i, j \le N$, let $J_{ij} := e_i J e_j$. If $J_{ij} \ne \{0\}$, then let $\lambda_{ij}$ be the length of $J_{ij}$ as an ($S_i$, $S_j$)-bimodule. If $J_{ij}=\{0\}$, then we take $\lambda_{ij}=0$.
\end{Def}

Recall that when $q_i$ and $q_j$ are powers of $p$, we define $q_i \otimes q_j$ to be the order of $\F_{q_i} \otimes_{\F_p} \F_{q_j}$. If $q_i = p^{d_i}$ and $q_j = p^{d_j}$, then $q_i \otimes q_j = p^{\lcm(d_i, d_j)}$. Note also that we have $M_{n_i}(q_i) \otimes_{\F_p} M_{n_j}(q_j) \cong M_{n_i n_j}(q_i \otimes q_j)$ for any positive integers $n_i$ and $n_j$. Thus, $|J_{ij}| = (q_i \otimes q_j)^{n_i n_j \lambda_{ij}}$.

We remark that the rings in parts (2) and (4) of Theorem \ref{thm:main} may occur as subrings of $R$, and can be described using the notation and parameters we have defined. If $n_k=1$ and $\lambda_{kk}=2$ for some $1 \le k \le N$, then $S_k \oplus J_{kk} \cong \F_{q_k} (+) \F_{q_k}^2$. If $1 \le i, j \le N$ with $i \ne j$, $n_j=1$, and $\lambda_{ij}=1$, then $S_i \oplus S_j \oplus J_{ij} \cong A(n_i,q_i,q_j)$ is a ring of AGL-type.


As we now show, any ($S_i$, $S_j$)-bimodule inside of $J$ is a two-sided ideal of $R$, and has an ideal complement in $J$.

\begin{lem}\label{lem:Peirce}
Let $1 \le i, j \le N$.
\begin{enumerate}[(1)]
\item Let $X \subseteq J$ be an ($S_i$, $S_j$)-bimodule. Then, $X \subseteq J_{ij}$, $X$ is a two-sided ideal of $R$, and $X$ has a two-sided ideal complement in $J$.
\item Let $I \subseteq J$ be a two-sided ideal of $R$. Then, $I$ has a two-sided ideal complement in $J$.
\item Let $I \subseteq J$ be a two-sided ideal of $R$ and let $Z = J \cap Z(R)$. For all $x \in J$, $S \oplus I = S^{1+x} \oplus I$ if and only if $x \in I + Z$. Hence, the number of conjugates of $S \oplus I$ in $R$ is equal to $|J:(I+Z)|$.
\item For each $1 \le i \le N$, $R$ has a residue ring isomorphic to $S_i \oplus J_{ii}$. For all $1 \le i, j \le N$ with $i \ne j$, $R$ has a residue ring isomorphic to $S_i \oplus S_j \oplus J_{ij}$.
\end{enumerate}
\end{lem}
\begin{proof}
(1) Since $e_i$ and $e_j$ are the identities of $S_i$ and $S_j$, respectively, $X = e_i X e_j \subseteq J_{ij}$. To show that $X$ is a two-sided ideal of $R$, let $r \in R$ and $x \in X$. Then, $r=s+y$ for some $s \in S$ and $y \in J$. Since $J^2= \{0\}$, $rx = sx$. Next, write $s = \sum_{k=1}^N s_k$, where each $s_k \in S_k$. Since the idempotents $e_1, \ldots, e_N$ are orthogonal, we have $se_i = s_i e_i$. So, $rx = sx = s(e_i x) = s_i x \in X$. The proof that $xr \in X$ is similar. Finally, since $S_i$ and $S_j$ are simple rings, $J$ is a semisimple ($S_i$, $S_j$)-bimodule. So, there is an ($S_i$, $S_j$)-bimodule $\whX \subseteq J$ such that $J = X \oplus \whX$, and by the work above, $\whX$ is also a two-sided ideal of $R$.

(2) Apply the Peirce decomposition to $I$ to get $I = \bigoplus_{i,j} e_i I e_j$. Then, each $e_i I e_j$ is an ($S_i$, $S_j$)-bimodule contained in $J$. By part (1), each $e_i I e_j$ has an ideal complement, and therefore $I$ does as well. 

(3) Note that since $J^2=\{0\}$, $(1+x)^{-1}=1-x$ for all $x \in J$. So, $s^{1+x} = (1-x)s(1+x) = s + sx-xs$ for all $s \in S$ and all $x \in J$. Certainly, if $x \in I +Z$, then $S^{1+x} \subseteq S \oplus I$. Conversely, assume that $x \in J$ is such that $S \oplus I = S^{1+x} \oplus I$. By (2), there is an ideal complement $\whI$ of $I$ in $J$. Express $x$ as $x = \alpha + \beta$, where $\alpha \in I$ and $\beta \in \whI$. Then, for all $s \in S$,
\begin{equation*}
sx - xs = (s\alpha - \alpha s) + (s\beta - \beta s),
\end{equation*}
and $sx-xs \in I$ because $s^{1+x} \in S \oplus I$. It follows that $s\beta - \beta s \in I \cap \whI = \{0\}$. Thus, $\beta$ is central in $R$ and $x \in I + Z$. 

(4) We will prove the claim regarding $S_i \oplus S_j \oplus J_{ij}$. The proof of the other statement is similar. Let $T = S_i \oplus S_j \oplus J_{ij}$, and let $\whS_{ij} = \bigoplus_{k \ne i, k \ne j} S_k$, which is a direct sum complement to $S_i \oplus S_j$ in $S$. Let $\whJ_{ij}$ be an ideal complement to $J_{ij}$ in $J$, which exists by part (2). Since the idempotents $e_1, \ldots, e_N$ are orthogonal, $\whS_{ij}$ and $T$ annihilate one another. Hence, $\whS_{ij} \oplus \whJ_{ij}$ is a two-sided ideal of $R$, and $R/(\whS_{ij} \oplus \whJ_{ij}) \cong T$.
\end{proof}

\begin{prop}\label{prop:idealcomplements}
Assume that $R$ is $\sigma$-elementary. Let $M$ be a maximal subring of $R$ that does not contain $J$. Then, $M$ is contained in every minimal cover of $R$.
\end{prop}
\begin{proof}
By Lemma \ref{lem:MaxSubringClassification}, $M = S' \oplus I$, where $S' \in \msS(R)$ and $I$ is a maximal subideal of $J$. By Lemma \ref{lem:Peirce}, $I$ has an ideal complement $\whI$ in $J$. Thus, $R = M \oplus \whI$. Since $R$ is $\sigma$-elementary, $M$ must be contained in every minimal cover of $R$ by Lemma \ref{lem:SigmaElementary}.
\end{proof}

Proposition \ref{prop:idealcomplements} (often in conjunction with Lemma \ref{lem:SigmaElementary}) will be used to produce lower bounds on the covering number of a $\sigma$-elementary ring. The remainder of this section contains various bounds on the covering numbers of $R$ and its subrings. 

\begin{lem}\label{lem:JcapZ}
Let $1 \le i, j \le N$ with $i \ne j$, and let $T = S_i \oplus S_j \oplus J_{ij}$.
\begin{enumerate}[(1)]
\item $J_{ij} \cap Z(T) = \{0\}$.

\item Let $I$ be a subideal of $J_{ij}$. Then, for all $x \in J_{ij}$, $(S_i \oplus S_j) \oplus I = (S_i \oplus S_j)^{1+x} \oplus I$ if and only if $x \in I$.

\item $\sigma(R) \le \sigma(T)$.
\end{enumerate}
\end{lem}
\begin{proof}
For (1), there is nothing to prove if $J_{ij} = \{0\}$, so assume there exists $x \in J_{ij} \setminus \{0\}$. Then, $e_i x = x$ and $x e_i = (x e_j)e_i = 0$. Hence, $x \in Z(T)$ if and only if $x = 0$. Part (2) now follows from part (1) and Lemma \ref{lem:Peirce}(3), and part (3) follows from Lemma \ref{lem:Peirce}(4).
\end{proof}

The next two lemmas are generalizations of \cite[Proposition 3.3]{SwartzWernerI}.

\begin{lem}\label{lem:aglupperbound}
Let $1 \le i, j \le N$ and let $q = q_i \otimes q_j$. Assume that $J_{ij} \ne \{0\}$, and let $T = S_i \oplus S_j \oplus J_{ij}$. Then, $\sigma(T) \le (q^{n_i n_j +1}-1)/(q-1)$.
\end{lem}
\begin{proof}
For convenience, let $S = S_i \oplus S_j$ and $J=J_{ij}$. First we argue that we may assume $J$ is a simple ($S_i$, $S_j$)-bimodule. Indeed, if $J$ is not simple, then let $X \subseteq J$ be a simple ($S_i$, $S_j$)-bimodule. By Lemma \ref{lem:Peirce}, $X$ has an ideal complement $\whX$ in $J$. Since $T/\whX \cong S_i \oplus S_j \oplus X$, we get $\sigma(T) \le \sigma(T/\whX) = \sigma(S_i \oplus S_j \oplus X)$. Hence, we may assume that $J$ is simple.

For each nonzero $x \in J$, define $C_S(x) := \{s \in S: sx = xs\}$. Then, each $C_S(x)$ is a subring of $S$. By Lemma \ref{lem:JcapZ}(1), $C_S(x) \subsetneq S$ for all nonzero $x \in J$, so $C_S(x) \oplus J$ is a proper subring of $T$ for all $x \ne 0$.

Let $\mcC = \msS(T) \cup \{C_S(x) \oplus J : x \in J, x \ne 0\}$. We will show that $\mcC$ is a cover of $T$. Suppose $t \in T \backslash \bigcup_{x \in J \setminus\{0\}} C_S(x) \oplus J$; that is, suppose $t \in s + J$ for some $s \in S \backslash \bigcup_{x \in J \setminus\{0\}} C_S(x)$.

We claim that $s+J = s^{1+J}$. We know that $s^{1+x} = s + sx-xs$ for all $x \in J$, so $s^{1+J} \subseteq s+J$. Furthermore, $s^{1+x_1}=s^{1+x_2}$ if and only if $s^{1+(x_1-x_2)} = s$. Since $s \notin C_S(x)$ for all $x \ne 0$, this means that $s^{1+x_1}=s^{1+x_2}$ if and only if $x_1=x_2$. Hence, $|s^{1+J}| = |J| = |s+J|$, and so $s+J = s^{1+J}$. Thus,
\begin{equation*}
t \in s + J = s^{1 + J} \subseteq \bigcup_{S' \in \msS(T)} S',
\end{equation*} 
and $\mathcal{C}$ is a cover of $T$. 

Finally, note that $C_S(cx) = C_S(x)$ whenever $c \in \F_q^\times$. This means that there are at most $(|J| - 1)/(q-1) = (q^{n_i n_j} - 1)/(q - 1)$ subrings $C_S(x)$ of $S$. Since $J$ is a simple $(S_i, S_j)$-bimodule, $|J|=q^{n_in_j}$. Therefore,
\begin{equation*}
\sigma(T) \le |J| + \frac{|J| - 1}{q-1} = \frac{q^{n_i n_j + 1} - 1}{q -1},
\end{equation*}
as required.
\end{proof}

\begin{lem}\label{lem:|J|lowerbound}
Let $1 \le i, j \le N$ such that $i \ne j$ and $n_j=1$, and let $q=q_i \otimes q_j$. If $J_{ij}$ is a simple ($S_i$, $S_j$)-bimodule and $R$ is $\sigma$-elementary, then $q^{n_i}+1 \le \sigma(R)$. If, in addition, $n_i = 1$, then $R = S_i \oplus S_j \oplus J_{ij}$ and $\sigma(R)=q+1$.
\end{lem}
\begin{proof}
Assume that $J_{ij}$ is simple and $R$ is $\sigma$-elementary. Let $T = S_i \oplus S_j \oplus J_{ij}$. Then, for each $x \in J_{ij}$, $(S_i \oplus S_j)^{1+x}$ is a maximal subring of $T$. By Lemma \ref{lem:JcapZ}, $J_{ij} \cap Z(T) = \{0\}$, so there are $|J_{ij}|=q^{n_i}$ distinct conjugates of $S_i \oplus S_j$ in $T$.

Let $\whS_{ij} = \bigoplus_{k \ne i, k \ne j} S_k$ and let $\whJ_{ij}$ be an ideal complement to $J_{ij}$ in $J$. Then, for any conjugate $S'$ of $S_i \oplus S_j$ in $T$, we have
\begin{equation*}
R = (S' \oplus \whS_{ij} \oplus \whJ_{ij}) \oplus J_{ij}.
\end{equation*}
Since $J_{ij}$ is simple, each subring $S' \oplus \whS_{ij} \oplus \whJ_{ij}$ is maximal. On the one hand, because $R$ is $\sigma$-elementary, each such subring is contained in every minimal cover of $R$ by Lemma \ref{lem:SigmaElementary}.
So, $q^{n_i} \le \sigma(R)$. On the other hand, the union of all the conjugates $S'$ cannot cover $T$. To see this, let $A \subseteq T$ be the intersection of all $q^{n_i}$ conjugates of $S_i \oplus S_j$. Then, the union of all the conjugates covers at most
\begin{equation*}
|J_{ij}|\big(|S_i \oplus S_j| - |A|\big) + |A| = |T| - |A|\big(|J|-1\big)
\end{equation*}
elements of $T$. It follows that the union of all the subrings $S' \oplus \whS_{ij} \oplus \whJ_{ij}$ cannot cover $R$. Hence, $q^{n_i}+1 \le \sigma(R)$.

Lastly, by Lemmas \ref{lem:JcapZ} and \ref{lem:aglupperbound}, $\sigma(R) \le \sigma(T) \le (q^{n_i + 1}-1)/(q-1)$. When $n_i=1$, this forces $\sigma(R)=q+1$ and $R=T$ because $R$ is $\sigma$-elementary.
\end{proof}

Recall the definitions for $\text{Irr}(p,d)$ and $\tau(q)$ stated prior to Theorem \ref{thm:main}. It is shown in \cite[Theorem 3.5]{Werner} that a direct sum $\bigoplus_{i=1}^t \F_q$ of copies of $\F_q$ is coverable if and only if $t \ge \tau(q)$.

\begin{lem}\label{lem:fieldbounds}
Let $q=p^d$, let $t \ge \tau(q)$, and let $T = \bigoplus_{i=1}^t \F_{q}$. If $d=1$, then $\sigma(T) = \tfrac{1}{2}(p^2+p)$. If $d \ge 2$, then $\sigma(T) \le q^2/(2d)$.
\end{lem}
\begin{proof}
As stated in Theorem \ref{thm:main}(1), we have
\begin{equation*}
\sigma(T) = \tau(q) \nu(q) + d \displaystyle\binom{\tau(q)}{2},
\end{equation*}
where $\nu(q)$ is equal to 1 if $d=1$, and is equal to the number of prime divisors of $d$ otherwise. Via this formula, it is clear that $\sigma(T)= p + \binom{p}{2} = \tfrac{1}{2}(p^2+p)$ when $d=1$. So, assume that $d \ge 2$. Then, $\nu(q) \le \tfrac{d}{2}$, and
\begin{equation}\label{eq:fieldbound}
\sigma(T) \le \tau(q) \cdot \tfrac{d}{2} + \tfrac{d}{2} \cdot \tau(q)\cdot(\tau(q)-1) = \tfrac{d}{2} \cdot \tau(q)^2.
\end{equation}
Now, it is well known that $|\text{Irr}(p,d)| \le p^d/d$, with equality if and only if $d=1$. When $d \ge 2$, $|\text{Irr}(p,d)| < p^d/d$, and hence $\tau(q) \le q/d$. Combining this with \eqref{eq:fieldbound} produces the desired upper bound for $\sigma(T)$.
\end{proof}

\section{Restrictions on \texorpdfstring{$J$}{J}}\label{sect:J}

Maintain the notation given at the start of Section \ref{sect:Peirce}. Our goal in this section is to prove that when $R$ is $\sigma$-elementary, most of the bimodules $J_{ij}$ must be zero. In essence, this means that a $\sigma$-elementary ring is ``almost'' semisimple. We will first examine the bimodules $J_{ij}$ with $i \ne j$, and then consider the ($S_k$, $S_k$)-bimodules $J_{kk}$.

\begin{thm}\label{thm:Jij=0}
Let $1 \le i, j \le N$ with $i \ne j$. Assume that $J_{ij} \ne \{0\}$ and $R$ is $\sigma$-elementary. Then, the following hold:
\begin{enumerate}[(1)]
\item $J_{k \ell} = \{0\}$ for all $1 \le k, \ell \le N$ such that $k \ne \ell$ and $(k, \ell) \ne (i,j)$. 
\item $J_{ij}$ is a simple ($S_i$, $S_j$)-bimodule.
\item At least one of $S_i$ or $S_j$ is a field.
\end{enumerate}
\end{thm}
\begin{proof}
Throughout, let $T_{ij} = S_i \oplus S_j \oplus J_{ij}$ and $q = q_i \otimes q_j$.

(1) Let $I_{ij}$ be a maximal subideal of $J_{ij}$. If $S'$ is a conjugate of $S_i \oplus S_j$ in $T_{ij}$, then $S' \oplus I_{ij}$ is a maximal subring of $T_{ij}$. Let $\whJ_{ij}$ be an ideal complement to $J_{ij}$ in $J$, and let 
\begin{equation*}
\whS_{ij} := \bigoplus_{m \ne i, m \ne j} S_m,
\end{equation*}
which is a direct sum complement to $S_i \oplus S_j$ in $S$. Then, for any conjugate $S'$ of $S_i \oplus S_j$ in $T_{ij}$,
\begin{equation}\label{eq:ijmax}
(S' \oplus \whS_{ij}) \oplus (I_{ij} \oplus \whJ_{ij})
\end{equation}
is a maximal subring of $R$. The number of such subrings is equal to the number of distinct maximal subrings of $T_{ij}$ of the form $S' \oplus I_{ij}$. Since the size of a simple $(S_i, S_j)$-bimodule is $q^{n_i n_j}$, by Lemma \ref{lem:JcapZ}(2), $R$ contains $|J_{ij}:I_{ij}| = q^{n_i n_j}$ maximal subrings of the form in \eqref{eq:ijmax}.

Suppose that $J_{k \ell} \ne \{0\}$ for some $(k, \ell) \ne (i, j)$. The arguments of the previous paragraph can be applied to $S_k \oplus S_\ell \oplus J_{k \ell}$. Employing the analogous notation and letting $q' = q_k \otimes q_\ell$, we see that $R$ contains $(q')^{n_k n_\ell}$ maximal subrings of the form
\begin{equation}\label{eq:klmax}
(S'' \oplus \whS_{k \ell}) \oplus (I_{k \ell} \oplus \whJ_{k \ell}),
\end{equation}
where $S''$ is a conjugate of $S_k \oplus S_\ell$. Moreover, no ring of the form of \eqref{eq:ijmax} can equal a ring of the form of \eqref{eq:klmax}, because $I_{ij} \oplus \whJ_{ij} \ne I_{k \ell} \oplus \whJ_{k \ell}$. By Proposition \ref{prop:idealcomplements},
\begin{equation}\label{eq:qq'}
q^{n_i n_j} + (q')^{n_k n_\ell} \le \sigma(R).
\end{equation}

Now, by Lemmas \ref{lem:JcapZ} and \ref{lem:aglupperbound},
\begin{equation}\label{eq:q}
\sigma(R) < \sigma(T_{ij}) \le \dfrac{q^{n_i n_j + 1}-1}{q-1}.
\end{equation}
Thus, by \eqref{eq:qq'} and \eqref{eq:q}, $q^{n_i n_j} + (q')^{n_k n_\ell} < (q^{n_i n_j + 1}-1)/(q-1)$. This gives
\begin{equation*}
(q')^{n_k n_\ell} < \dfrac{q^{n_i n_j + 1}-1}{q-1} - q^{n_i n_j} = \dfrac{q^{n_i n_j }-1}{q-1} < q^{n_i n_j}.
\end{equation*}
However, we also have $\sigma(R) < \sigma(S_k \oplus S_\ell \oplus J_{k \ell})$. Working as above, this leads to $q^{n_i n_j} < (q')^{n_k n_\ell}$, a contradiction.

(2) Suppose that $\lambda_{ij} \ge 2$. Then, $R$ still contains maximal subrings as in \eqref{eq:ijmax}. But, $J_{ij}$ contains at least two distinct maximal subideals, so we have at least two choices for $I_{ij}$. Thus, $R$ contains at least $2 q^{n_i n_j}$ maximal subrings of the form of \eqref{eq:ijmax}. Appealing again to Lemma \ref{lem:aglupperbound}, we have
\begin{equation*}
2 q^{n_i n_j} \le \sigma(R) < \sigma(T_{ij}) \le \dfrac{q^{n_i n_j + 1}-1}{q-1}.
\end{equation*}
But then, $2 q^{n_i n_j}(q-1) < q^{n_i n_j + 1}-1$, which is impossible. Hence, $\lambda_{ij}=1$ and $J_{ij}$ is simple.

(3) Suppose that both $S_i$ and $S_j$ are noncommutative, i.e., that $n_i \ge 2$ and $n_j \ge 2$. As in part (1), we obtain $q^{n_i n_j} \le \sigma(R)$ by considering maximal subrings of the form in \eqref{eq:ijmax}. Without loss of generality, assume that $n_i \le n_j$. Since $R$ is $\sigma$-elementary, $\sigma(R) < \sigma(R/J) \le \sigma(S_i)$. But,
\begin{equation*}
\sigma(S_i) = \sigma(M_{n_i}(q_i)) < q^{n_i^2} \le q^{n_i n_j}
\end{equation*}
a contradiction.  Thus, at least one of $S_i$ or $S_j$ must be a field.

\end{proof}

\begin{Rem}
\label{rem:mutmut1}
Theorem \ref{thm:Jij=0} holds mutatis mutandis for unital maximal subrings and unital coverings of $R$. This fact will be used later in Section \ref{sect:othermain} when we examine $\sigma_u$-elementary rings.
\end{Rem}

It remains possible that $R$ is $\sigma$-elementary and $J_{ii} \ne \{0\}$ for some $i$. Indeed, this case cannot be excluded entirely, because $\F_q (+) \F_q^2$ is $\sigma$-elementary and has a nonzero radical. However, we can restrict the possibilities for $J_{ii}$ in this case as well. First, we show that if $R$ is $\sigma$-elementary and $S_i$ is noncommutative, then $J_{ii} = \{0\}$.

\begin{lem}\label{lem:Si-Si}
Let $1 \le i \le N$ and let $T = S_i \oplus J_{ii}$. Assume that $n_i \ge 2$ and $J_{ii} \ne \{0\}$. Then, $|J_{ii} \cap Z(T)| = q^{\lambda_{ii}}$.
\end{lem}
\begin{proof}
For convenience, let $S_i=M_n(q)$, $J=J_{ii}$, and $\lambda=\lambda_{ii}$. As in \cite[Section 11.8]{Pierce}, $J$ decomposes as a direct sum $\bigoplus_{k=1}^\lambda X_k$ of simple ($M_n(q)$, $M_n(q)$)-bimodules $X_1, \ldots, X_\lambda$. Each $X_k$ is equivalent to a simple $M_{n^2}(q)$ module---hence has order $q^{n^2}$---and the isomorphism type of $T$ is determined entirely by $\lambda$. Thus, we can represent each $X_k$ (additively) as a copy of $M_n(q)$, and represent $T$ in block matrix form as
\begin{equation*}
T = \left\{\begin{pmatrix} A & x_1 & \cdots & x_\lambda\\ 0 & A & \cdots & 0\\ \vdots & & \ddots & & \\ 0 & \cdots & 0 & A \end{pmatrix} : A \in M_n(q), x_k \in X_K \text{ for } 1 \le k \le \lambda \right\}.
\end{equation*}
Via this representation, we see that $x \in J \cap Z(T)$ if and only if $x=\sum_k x_k$ and each $x_k$ is an $n \times n$ scalar matrix. Thus, $|J \cap Z(T)| = q^\lambda$.
\end{proof}

\begin{prop}\label{prop:Jii=0}
Let $1 \le i \le N$ and let $T = S_i \oplus J_{ii}$. Assume that $n_i \ge 2$. 
\begin{enumerate}[(1)]
\item $T$ is coverable, and $\sigma(T) = \sigma(M_{n_i}(q_i))$.
\item If $R$ is $\sigma$-elementary, then $J_{ii} = \{0\}$.
\end{enumerate}
\end{prop}
\begin{proof}
As in earlier proofs, we will drop some subscripts for readability. Let $n=n_i$, $q=q_i$, and $\lambda=\lambda_i$.

(1) The ring $T$ is coverable because it is noncommutative. For the covering number of $T$, we use induction on $\lambda$. If $\lambda=0$, then $T = M_n(q)$ and we are done. So, assume that $\lambda \ge 1$ and that the covering number of $T':=M_n(q) \oplus J'$ is $\sigma(M_n(q))$ whenever $J'$ is an ($M_n(q)$, $M_n(q)$)-bimodule of length $\lambda-1$. 

Consider $M:=S_i \oplus I$, where $I$ is a maximal $T$-subideal of $J_{ii}$. Note that $I$ has length $\lambda-1$, so $\sigma(M) = \sigma(M_n(q))$ and $|J_{ii}:I| = q^{n^2}$. On the one hand, $I$ has an ideal complement in $J_{ii}$, so $M$ occurs as a residue ring of $T$ and $\sigma(T) \le \sigma(M)$. On the other hand, by Lemmas \ref{lem:Peirce}(3) and \ref{lem:Si-Si}, the number of conjugates of $M$ in $T$ is
\begin{equation*}
|J_{ii}:(I+(J_{ii}\cap Z(T)))| = |J_{ii}:I| \cdot \frac{|I \cap Z(T)|}{|J_{ii} \cap Z(T)|} = q^{n^2} \cdot \frac{q^{\lambda-1}}{q^\lambda} = q^{n^2-1}.
\end{equation*}
By \cite[Lemma 3.4]{SwartzWernerI}, $\sigma(M) \le q^{n^2/2}$, so any minimal cover of $T$ excludes some conjugate of $M$. Hence, $\sigma(M) \le \sigma(T)$ by Lemma \ref{lem:SigmaElementary}(1).

(2) Suppose that $R$ is $\sigma$-elementary, but $J_{ii} \ne \{0\}$. Let $I$ be a maximal subideal of $J_{ii}$, let $\whJ_{ii}$ be an ideal complement for $J_{ii}$ in $J$, and let $\whS_i = \bigoplus_{k \ne i} S_k$. Then, $S_i \oplus I$ is maximal in $T$, so
\begin{equation*}
M:= (S_i \oplus \whS_i) \oplus (I \oplus \whJ_{ii})
\end{equation*}
is a maximal subring of $R$. As noted earlier, the number of conjugates of $S_i \oplus I$ in $T$ is $q^{n^2-1}$. So, $R$ contains $q^{n^2-1}$ such maximal subrings $M$. Since 
\begin{equation*}
\sigma(R) < \sigma(S_i) = \sigma(M_n(q)) < q^{n^2-1},
\end{equation*}
every minimal cover of $R$ excludes at least one $M$. By Lemma \ref{lem:SigmaElementary}(1), $\sigma(M) \le \sigma(R)$. But, $M$ occurs as a residue ring of $R$, so $\sigma(R) < \sigma(M)$, a contradiction.
\end{proof}

Next, we consider $J_{kk}$ when $S_k$ is a field. We will prove that for $R$ to be $\sigma$-elementary in this case, it is necessary that $\lambda_{kk} \ge 2$; and, if $J_{ij} \ne \{0\}$ for some $i \ne j$, then $J_{kk} = \{0\}$ whenever $S_k$ is isomorphic to $S_i$ or $S_j$.

\begin{lem}\label{lem:Jkk}
Let $1 \le k \le N$. Assume that $S_k$ is a field and that $J_{kk}$ is a simple ($S_k$, $S_k$)-bimodule. 
\begin{enumerate}[(1)]
\item $R$ contains a unique maximal subring that does not contain $J_{kk}$.
\item $\sigma(R) = \sigma(R/J_{kk})$ and $R$ is not $\sigma$-elementary.
\end{enumerate}
\end{lem}
\begin{proof}
(1) First, we note that $J_{kk} \subseteq Z(R)$. Indeed, suppose that $x \in J_{kk}$ and $s \in S$. Write $s = \sum_{i=1}^N s_i $, where each $s_i \in S_i$. Then, $sx = s_kx$. But, $S_k \oplus J_{kk}$ is commutative, so $s_k x = x s_k = xs$.

Next, let $\whJ_{kk}$ be an ideal complement for $J_{kk}$ in $J$, and let $T = S \oplus J_{kk}$. Then, $T$ is a maximal subring of $R$ because $J_{kk}$ is simple. Let $T'$ be any maximal subring of $R$ that does not contain $J_{kk}$. By Lemma \ref{lem:MaxSubringClassification}, $T' = S^{1+y} \oplus I$ for some $y \in J$ and some maximal subideal $I$ of $J$.

We claim that $I = \whJ_{kk}$. To see this, apply the Peirce decomposition to $I$ to get
\begin{equation*}
I = (e_k I e_k) \oplus \Big( \bigoplus_{(i,j) \ne (k,k)} e_i I e_j\Big) = (I \cap J_{kk}) \oplus (I \cap \whJ_{kk}).
\end{equation*}
Since $J_{kk}$ is simple and $J_{kk} \not\subseteq I$, $I \cap J_{kk} = \{0\}$. Then, by maximality, $I = \whJ_{kk}$.

Now, let $\alpha \in J_{kk}$ and $\beta \in \whJ_{kk}$ be such that $y=\alpha + \beta$. Since $\alpha$ is central, conjugation by $1+\alpha$ fixes elements of $S$ pointwise. Thus,
\begin{equation*}
S^{1+y} = (S^{1+\alpha})^{1+\beta} = S^{1+\beta} \subseteq S \oplus \whJ_{kk}
\end{equation*}
and so $T' = T$.

(2) As in part (1), let $\whJ_{kk}$ be the ideal complement for $J_{kk}$ in $J$ and let $T:=S \oplus \whJ_{kk}$ be the unique maximal subring of $R$ that does not contain $J_{kk}$. Let $\mcC$ be a minimal cover of $R$, and let $A_1, \ldots, A_m$ be all the maximal subrings in $\mcC$ that contain $J_{kk}$.

Note that $R = T \oplus J_{kk}$. Let $a \in T$, let $x \in J_{kk} \setminus \{0\}$, and consider $a+x$. Since $S_k \subseteq T$ and $J_{kk}$ is simple, if $x \in T$ then $S_k \cdot x = J_{kk} \subseteq T$, a contradiction. So, $x \notin T$, and hence $a+x \notin T$ as well. Consequently, $a+x \in \bigcup_{\ell=1}^m A_\ell$. This holds for all $a \in T$, so $T+x \subseteq \bigcup_{\ell=1}^m A_\ell$. Recalling that $J_{kk} \subseteq A_\ell$ for each $\ell$, we see that $\{A_1/J_{kk}, \ldots, A_m/J_{kk}\}$ is a cover of $R/J_{kk} \cong T$. Thus, $\sigma(R/J_{kk}) \le m \le \sigma(R)$, and so in fact $\sigma(R) = \sigma(R/J_{kk})$.
\end{proof}

\begin{Rem}
\label{rem:mutmut2}
 We note that Lemma \ref{lem:Jkk} holds mutatis mutandis for unital maximal subrings and unital coverings of $R$. As with Theorem \ref{thm:Jij=0}, this will become relevant in Section \ref{sect:othermain}.
\end{Rem}

\begin{prop}\label{prop:Jkk}
Let $1 \le i, j \le N$ such that $i \ne j$. Assume that $J_{ij}$ is a simple ($S_i$, $S_j$)-bimodule, $S_j$ is a field, and $R$ is $\sigma$-elementary. Then, for all $1 \le k \le N$, if $S_k \cong S_j$, then $J_{kk} = \{0\}$.
\end{prop}
\begin{proof}
Assume $1 \le k \le N$ is such that $S_k \cong S_j$, and suppose that $J_{kk} \ne \{0\}$. Since $R$ is $\sigma$-elementary, $J_{kk}$ cannot be simple by Lemma \ref{lem:Jkk}. So, $\lambda_{kk} \ge 2$. 

Now, $S_k \oplus J_{kk}$ is a commutative ring. Since $\lambda_{kk} \ge 2$ and $S_k \cong \F_{q_j}$, the ring $S_k \oplus J_{kk}$ has $\F_{q_j} (+) \F_{q_j}^2$ as a residue ring. By \cite[Theorem 4.8(2)]{SwartzWerner},
\begin{equation*}
\sigma(S_k \oplus J_{kk}) \le \sigma(\F_{q_j} (+) \F_{q_j}^2) = q_j+1.
\end{equation*}
Moreover, $S_k \oplus J_{kk}$ occurs as a residue ring of $R$ by Lemma \ref{lem:Peirce}(4), so $\sigma(R) < q_j+1$. However, letting $q=q_i \otimes q_j$ and applying Lemma \ref{lem:|J|lowerbound} yields
\begin{equation*}
q_j +1 \le q^{n_i} + 1 \le \sigma(R),
\end{equation*}
which is a contradiction. Thus, $J_{kk} = \{0\}$.
\end{proof}

Based on our work so far, it remains possible that $J_{ij} \ne \{0\}$ and $J_{kk} \ne \{0\}$ for some field $S_k$ that is isomorphic to neither $S_i$ nor $S_j$. Ultimately, this case is ruled out by our classification of $\sigma$-elementary rings in Theorem \ref{thm:main}. Rather than attack this case directly, however, we merely note that $S_k \oplus J_{kk}$ is always a commutative ring. The conclusion that $J_{kk}=\{0\}$ when $J_{ij} \ne \{0\}$ will be reached later by referencing the classification of commutative $\sigma$-elementary rings from \cite{SwartzWerner}.

\section{Restrictions on direct summands of \texorpdfstring{$S$}{S}}\label{sect:summands}

Continue to use the notations of Sections \ref{sect:Peirce} and \ref{sect:J}. At this point, we know that in a $\sigma$-elementary ring, the radical has a restricted structure. There is at most one pair $(i, j)$ with $i \ne j$ and $J_{ij} \ne \{0\}$, and if $J_{kk} \ne \{0\}$, then $S_k$ must be a field not isomorphic to $S_i$ or $S_j$. We can also place restrictions on the direct summands of $S$. In Proposition \ref{prop:onecopy}, we show that if $S_i$ is noncommutative, then $S_j \not\cong S_i$ for all $j \ne i$. A closely related result holds for a field $S_i$ (Proposition \ref{prop:fieldcopies}): if $J_{ij} \ne \{0\}$ with $i \ne j$ and $S_i$ is a field, then $S_k \not\cong S_i$ for all $k \ne i, j$.

We begin by recalling some results from \cite{PeruginelliWerner}. 

\begin{lem}\label{lem:linked}
Let $A = M_n(q)$.
\begin{enumerate}[(1)]
\item \cite[Proposition 4.2]{PeruginelliWerner} Let $\phi$ be an $\F_p$-automorphism of $A$. Then, $M:=\{a + \phi(a) : a \in A\}$ is a maximal subring of $A \oplus A$, and $M \cong A$.
\item \cite[Lemma 5.5]{PeruginelliWerner} If $n \ge 2$, then the number of $\F_p$-automorphisms of $A$ is greater than $\sigma(A)$.
\end{enumerate}
\end{lem}

\begin{lem}\label{lem:isolated}
Let $1 \le i \le N$. Assume that $J_{ij}=J_{ji}=\{0\}$ for all $j \ne i$. Then, $S_i \oplus J_{ii}$ is a two-sided ideal of $R$.
\end{lem}
\begin{proof}
Let $R_1 := S_i \oplus J_{ii}$ and 
\begin{equation*}
R_2 := \Big(\bigoplus_{j \ne i} S_j \Big) \oplus \Big(\bigoplus_{(j,k) \ne (i,i)} J_{jk}\Big),
\end{equation*}
so that $R = R_1 \oplus R_2$. By assumption, $J_{ij} = \{0\} = J_{ji}$ for all $j \ne i$. So, 
\begin{equation*}
R_2 = \Big(\bigoplus_{j \ne i} S_j \Big) \oplus \Big(\bigoplus_{j \ne i, k \ne i} J_{jk}\Big).
\end{equation*}
This means that $r_1 r_2 = 0 = r_2 r_1$ for all $r_1 \in R_1$ and $r_2 \in R_2$. It follows that $R_1$ is a two-sided ideal of $R$.
\end{proof}

\begin{prop}\label{prop:onecopy}
Let $1 \le i, j \le N$ such that $n_i \ge 2$. If $R$ is $\sigma$-elementary, then $S_i \not\cong S_j$ for all $j \ne i$.
\end{prop}
\begin{proof}
Assume that $R$ is $\sigma$-elementary, but there exists $j \ne i$ such that $S_j \cong S_i$. By Theorem \ref{thm:Jij=0}, there is at most one pair $(k,\ell)$ with $k \ne \ell$ and $J_{k \ell} \ne \{0\}$, and we must have $J_{ij}=J_{ji}=\{0\}$ because $S_i$ and $S_j$ are noncommutative. Also, $J_{ii} = J_{jj} = \{0\}$ by Proposition \ref{prop:Jii=0}. It follows that either $J_{ik} = J_{ki}=\{0\}$ for all $1 \le k \le N$, or $J_{jk} = J_{kj}=\{0\}$ for all $1 \le k \le N$. Assume without loss of generality that the latter condition holds. Then, $S_j$ is a two-sided ideal of $R$ by Lemma \ref{lem:isolated}.

For any $\F_p$-algebra isomorphism $\phi: S_i \to S_j$, let $M_\phi$ be the maximal subring $M_\phi:= \{a + \phi(a) : a \in S_i\}$ of $S_i \oplus S_j$. Note that $M_{\phi_1} = M_{\phi_2}$ if and only if $\phi_1 = \phi_2$. Let
\begin{equation*}
T:= \Big(\bigoplus_{k \ne i, k \ne j} S_k\Big) \oplus M_\phi \oplus J,
\end{equation*}
which is maximal subring of $R$. Furthermore, since $M_\phi \cong S_i \cong S_j$, we have 
\begin{equation*}
T \cong \Big(\bigoplus_{k \ne j} S_k\Big) \oplus J \cong R/S_j,
\end{equation*}
and so $\sigma(R) < \sigma(R/S_j) = \sigma(T)$. 

Now, let $m$ be the number of maximal subrings of $R$ isomorphic to $T$. Then, $m$ is greater than or equal to the number of choices for $M_\phi$, which by Lemma \ref{lem:linked}(2) is greater than $\sigma(M_{n_i}(q_i)) = \sigma(S_i)$. Since $R$ has a residue ring isomorphic to $S_i$, we have $\sigma(R) < \sigma(S_i) < m$. Hence, any minimal cover of $R$ omits some maximal subring $T'$ isomorphic to $T$. Applying Lemma \ref{lem:SigmaElementary}(1), we have $\sigma(T) = \sigma(T') \le \sigma(R)$, which is a contradiction.
\end{proof}

Ruling out fields isomorphic to $S_j$ is more difficult. By \cite[Theorem 3.5]{Werner}, $T=\bigoplus_{\ell = 1}^t \F_q$ is coverable if and only if $t \ge \tau(q)$, where $\tau(q)$ is as defined in the introduction. When $T$ is not coverable, it is generated (as a ring) by a single element $\alpha = \sum_{\ell = 1}^t \alpha_\ell$. By \cite[Proposition 3.4]{Werner}, $\alpha$ generates $T$ if and only if each $\alpha_\ell$ is a primitive element for $\F_q$ over $\F_p$ (if $q=p$, then any nonzero element of $\F_p$ is primitive) and each $\alpha_\ell$ has a distinct minimal polynomial over $\F_p$.

We also require knowledge about the maximal subrings of $S$. These subrings were classified in \cite{PeruginelliWerner}.

\begin{Def}\label{def:pi1pi2}\cite[Definition 4.3]{PeruginelliWerner}
Let $S = \bigoplus_{i = 1}^N S_i$, where each $S_i$ is a finite simple ring of characteristic $p$. Let $M$ be a maximal subring of $S$. We say that $M$ is of \emph{Type $\Pi_1$} if there exists an index $i$ and a maximal subring $M_i$ of $R_i$ such that
\begin{equation*}
M = S_1 \oplus \cdots \oplus S_{i-1} \oplus M_i \oplus S_{i+1} \oplus \cdots \oplus S_N.
\end{equation*}
We say that $M$ is of \emph{Type $\Pi_2$} if there exist indices $i < j$ and an $\F_p$-algebra isomorphism $\phi: S_i \to S_j$ such that
\begin{equation*}
M = \Big\{\sum_{k=1}^N s_k \in S : s_k \in S_k \text{ for each } 1 \le k \le N \text{ and } s_j=\phi(s_i)\Big\};
\end{equation*}
here, there are no restrictions on the summands $s_k$ for $k \ne i, j$. In a Type $\Pi_2$ maximal subring, we say that $S_i$ and $S_j$ are \emph{linked} by $\phi$.
\end{Def}

\begin{lem}\label{lem:semisimplemax}\cite[Theorem 4.5]{PeruginelliWerner}
Any maximal subring of $S$ is Type $\Pi_1$ or Type $\Pi_2$.
\end{lem}

\begin{prop}\label{prop:fieldcopies}
Let $1 \le i, j \le N$ with $i\ne j$. Assume that $J_{ij} \ne \{0\}$ and $S_j$ is a field. If $R$ is $\sigma$-elementary, then $S_k \not\cong S_j$ for all $k \ne i, j$.
\end{prop}
\begin{proof}
Assume that $R$ is $\sigma$-elementary. By Theorem \ref{thm:Jij=0}, $J_{ij}$ is a simple ($S_i$, $S_j$)-bimodule. Let $\msT:=\{1 \le k \le N : k \ne i, k \ne j, S_k \cong S_j\}$. Suppose that $\msT \ne \varnothing$, let $t = |\msT|$, and let $T = \bigoplus_{k \in \msT} S_k$. From Theorem \ref{thm:Jij=0}, we know that $J_{k\ell}=J_{\ell k} = \{0\}$ whenever $k \in \msT$ and $\ell \ne k$. Moreover, by Proposition \ref{prop:Jkk}, $J_{kk} =\{0\}$ for all $k \in \msT$. Hence, by Lemma \ref{lem:isolated}, each $S_k$---and therefore $T$---is a two-sided ideal of $R$.

Let $q_j = p^{d_j}$ and $q = q_i \otimes q_j$. If $n_i=1$, then $R=S_i \oplus S_j \oplus J_{ij}$ by Lemma \ref{lem:|J|lowerbound}, and there is nothing to prove because $\msT=\varnothing$. So, we may assume that $n_i \ge 2$.

We break the proof into two cases, depending on whether $t < \tau(q_j)-1$ or $t \ge \tau(q_j)-1$. The latter case is easier, so we will deal with that first.

When $t \ge \tau(q_j)-1$, $S_j \oplus T \cong \bigoplus_{i=1}^{t+1} \F_{q_j}$ is coverable, and $\sigma(R) < \sigma(S_j \oplus T)$ since $S_j \oplus T$ occurs as a residue ring of $R$. By Lemma \ref{lem:fieldbounds},
\begin{equation}\label{eq:taubound}
\sigma(R) < \begin{cases} \tfrac{1}{2}(p^2+p), & d_j=1\\  q_j^2/(2d_j), &d_j \ge 2 \end{cases}.
\end{equation}
However, by Lemma \ref{lem:|J|lowerbound}, $q_j^2+1 \le q^{n_i} + 1 \le \sigma(R)$. This is not compatible with \eqref{eq:taubound}, since neither $p^2+1 < \tfrac{1}{2}(p^2+p)$ nor $q_j^2+1 < q_j^2/(2d_j)$ is a valid inequality.

For the remainder of the proof, assume that $t < \tau(q_j)-1$. Define $\whT$ to be $\whT:=\bigoplus_{\ell \notin \msT} S_\ell$. Then, $R = (\whT \oplus J) \oplus T$. 

We will sort the maximal subrings of $R$ into three classes $\mcC_1$, $\mcC_2$, and $\mcC_3$. Let $\mcC_1$ be the collection of all maximal subrings of $R$ that contain $T$. Note that these maximal subrings are all inverse images of maximal subrings of $R/T$. 

Next, suppose $M \subseteq R$ is maximal with $T \not\subseteq M$. If $J \not\subseteq M$, then $M = S' \oplus I$ for some $S' \in \msS(R)$ and some maximal subideal $I \subseteq J$. However, as noted earlier, $J_{k\ell} = J_{\ell k} = \{0\}$ for all $1 \le \ell \le N$ and all $k \in \msT$. So, the action of $T$ on $J$ is trivial. It follows that $T \subseteq S'$ for all $S' \in \msS(R)$, and hence $J$ must be contained in $M$. 

Since $J \subseteq M$ and $T \not\subseteq M$, $M$ is the inverse image of a maximal subring of $R/J$ that does not contain $T$. In light of Lemma \ref{lem:semisimplemax}, we can partition such subrings into two classes. Use a bar to denote passage from $R$ to $R/J$. Let $\mcC_2$ be the set of all subrings of the form $\whT \oplus J \oplus M'$, where $M'$ is maximal in $T$. Subrings in $\mcC_2$ are those where $\overline{S_j}$ is not linked to any $\overline{S_k}$, $k \in \msT$. Finally, let $\mcC_3$ be the collection of maximal subrings where $\overline{S_j}$ is linked to some $\overline{S_k}$ with $k \in \msT$.

Now, since $t < \tau(q_j)$, $T$ is not coverable. Next, we will show that for all $a \in \whT \oplus J$, we can find $\alpha \in T$ such that $\alpha$ generates $T$ (as a ring), and $a+\alpha$ is not contained in a maximal subring in either $\mcC_2$ or $\mcC_3$. This is easy to see with regard to $\mcC_2$: for any $a$ and $\alpha$ as above, if $a+\alpha \in M \in \mcC_2$, then $\alpha \in M'$ for some maximal subring $M'$ of $T$. But, this is impossible because $\alpha$ generates $T$.

For $\mcC_3$, the argument is more delicate. Let $\whS_j$ be an additive complement to $S_j$ in $\whT$. Write $a$ as $a=s+b$, where $s \in S_j$ and $b \in \whS_j \oplus J$. Since $t < \tau(q_j)-1$, $t+1$ is less than or equal to the number of irreducible polynomials in $\F_p[x]$ of degree $d_j$. So, for each $k \in \msT$, we can find $\alpha_k \in S_k$ such that $\alpha_k$ generates $S_k$; each $\alpha_k$ has a distinct minimal polynomial over $\F_p$; and no $\alpha_k$ has the same minimal polynomial as $s$. Let $\alpha = \sum_{k \in \msT} \alpha_k$. Then, for every $k$ and every automorphism $\phi$ of $\F_{q_j}$, $\phi(s) \ne \alpha_k$. We conclude that $a + \alpha$ cannot lie in any subring where $\overline{S_j}$ is linked to some $\overline{S_k}$.

We have shown that for every $a \in \whT \oplus J$, there exists $\alpha \in T$ such that $a + \alpha$ lies in a maximal subring in class $\mcC_1$. Now, consider a minimal cover $\mcC$ of $R$. Let $\mcA = \mcC \cap \mcC_1$. By our work above, $\mcA \ne \varnothing$, because elements of the form $a+\alpha$ can only lie in maximal subrings in $\mcC_1$. Furthermore, the subrings in $\mcA$ form a cover of $R/T$, so $\sigma(R/T) \le |\mcA| \le \sigma(R)$. This contradicts the fact that $R$ is $\sigma$-elementary, and we are done.\end{proof}

\section{Proof of Theorem \ref{thm:main}}\label{sect:mainproof}

We now have all the ingredients necessary to prove Theorem \ref{thm:main}. Let $R$ be a $\sigma$-elementary ring with unity. The idea of the proof is to write $R$ as a direct sum of subrings $R = R_1 \oplus R_2 \oplus R_3$ in such a way that all maximal subrings of $R$ respect the decomposition. That is, any maximal subring $M$ of $R$ is equal to one of
\begin{equation}\label{eq:max3}
M_1 \oplus R_2 \oplus R_3, \; R_1 \oplus M_2 \oplus R_3, \; \text{or} \; R_1 \oplus R_2 \oplus M_3,
\end{equation}
where $M_i$ is maximal in $R_i$. Then, $\sigma(R) = \min\{\sigma(R_1), \sigma(R_2), \sigma(R_3)\}$ by Lemma \ref{lem:AMM2.2}, and the problem is reduced to a small number of possible cases.

\begin{prop}\label{prop:R1R2R3}
Assume $R$ is a finite unital ring of characteristic $p$ containing subrings $R_1$, $R_2$, and $R_3$ that satisfy all of the following properties:
\begin{itemize}
\item $R = R_1 \oplus R_2 \oplus R_3$.
\item $R_1 \cong A(n,q_1,q_2)$ for some $n \ge 2$.
\item $R_2 \cong \bigoplus_{j \in \mcI} M_{n_j}(q_j)$, where $\mcI$ is some finite index set, $n_j \ge 2$ for each $j$, and $M_{n_j}(q_j) \not\cong M_n(q_1)$ for all $j$.
\item $R_3$ is commutative, and if a field $F$ occurs as a direct summand of $R_3 / \msJ(R_3)$, then $F \not\cong \F_{q_2}$.
\end{itemize}
Then, $\sigma(R) = \min\{\sigma(R_1), \sigma(R_2), \sigma(R_3)\}$.
\end{prop}
\begin{proof}
It suffices to show that every maximal subring of $R$ decomposes as in \eqref{eq:max3}. Then, Lemma \ref{lem:AMM2.2} can be used to reach the desired conclusion.

Let $J = \msJ(R)$. For $i=1,2,3$, let $J_i = \msJ(R_i)$, and fix a semisimple complement $S_i$ for $J_i$ in $R_i$. Note that $J_2 = \{0\}$ and $S_2=R_2$ because $R_2$ is semisimple. So, $J=J_1 \oplus J_3$ and $S := S_1 \oplus R_2 \oplus S_3$ is a semisimple complement to $J$ in $R$.

Let $M$ be a maximal subring of $R$. Assume first that $J \not\subseteq M$. By Lemma \ref{lem:MaxSubringClassification}, $M=S' \oplus I$, where $S'$ is a conjugate of $S$ and $I$ is a maximal subideal of $J$. Any such $I$ decomposes as either $I=I_1 \oplus J_3$ or $I=J_1 \oplus I_3$, where $I_1$ (respectively, $I_3$) is maximal in $J_1$ (respectively, $J_3$). Moreover, the conjugates of $S$ have a similar form. Indeed, $S_1 \oplus S_2$ has a trivial action on $J_3$, and likewise for the action of $S_2 \oplus S_3$ on $J_1$. So, given $x = x_1 + x_3 \in J$ with $x_1 \in J_1$ and $x_3 \in J_3$, we have $S^{1+x} = S_1^{1+x_1} \oplus R_2 \oplus S_3^{1+x_3}$. 

With these observations in mind, we see that if $I = I_1 \oplus J_3$, then 
\begin{align*}
M = S' \oplus I =(S_1' \oplus I_1) \oplus R_2 \oplus (S_3' \oplus J_3)
\end{align*}
for some conjugate $S_1'$ of $S_1$ and some conjugate $S_3'$ of $S_3$. In this case, $S_1' \oplus I_1$ is maximal in $R_1$, and $S_3' \oplus J_3 = R_3$. Thus, $M$ has the form of \eqref{eq:max3}. The same steps show that this also holds when $I = J_1 \oplus I_3$.

Now, assume that $J \subseteq M$. Then, $M = T \oplus J$ for some maximal subring $T$ of $S$. By Lemma \ref{lem:semisimplemax}, $T$ is either of Type $\Pi_1$ or Type $\Pi_2$. If $T$ is of Type $\Pi_1$, then $T$ has one of the forms
\begin{equation*}
T_1 \oplus S_2 \oplus S_3, \; S_1 \oplus T_2 \oplus S_3, \; \text{ or } \; S_1 \oplus S_2 \oplus T_3,
\end{equation*}
where $T_i$ is maximal in $S_i$ for $i=1,2,3$, and $M$ decomposes as in \eqref{eq:max3}. On the other hand, if $T$ is of Type $\Pi_2$, then (modulo $J$) two simple summands $A$ and $A'$ of $S$ are linked by an $\F_p$-algebra isomorphism. If $A$ is a field, then both $A$ and $A'$ are subrings of $S_3$, because no simple summand of $S_3$ is isomorphic to $\F_{q_2}$. In this case, $T = S_1 \oplus S_2 \oplus T_3$. Similarly, if $A$ is noncommutative, then both $A$ and $A'$ are subrings of $S_2$, because no simple summand of $S_2$ is isomorphic to $M_n(q_1)$. Hence, $T = S_1 \oplus T_2 \oplus S_3$. Thus, in all cases $M$ has the form of \eqref{eq:max3}, and we are done.
\end{proof}

\begin{myproof}[of Theorem \ref{thm:main}]
We may assume that $R$ is finite, has characteristic $p$, and $J:=\msJ(R)$ satisfies $J^2 =\{0\}$. We will use the notation established at the beginning of Section \ref{sect:Peirce}. By Theorem \ref{thm:Jij=0}, there is at most pair $(i, j)$ such that $1 \le i, j \le N$, $i \ne j$, and $J_{ij} \ne \{0\}$. If such a pair exists, then $J_{ij}$ is simple and one of $S_i$ or $S_j$ is a field. Without loss of generality, assume that $n_j=1$. Then, $S_i \oplus S_j \oplus J_{ij} \cong A(n_i, q_i, q_j)$. If $n_i = 1$, then $R = S_i \oplus S_j \oplus J_{ij}$ by Lemma \ref{lem:|J|lowerbound} and we are done. So, assume that $n_i \ge 2$.

If $1 \le k \le N$ is such that $S_k$ is noncommutative, then $J_{kk}=\{0\}$ by Proposition \ref{prop:Jii=0}, and by Proposition \ref{prop:onecopy}, $S_\ell \not\cong S_k$ for all $\ell \ne k$. In particular, $J_{ii} = \{0\}$ and $S_k \not\cong S_i$ when $k \ne i$. Likewise, $J_{jj} = \{0\}$ by Proposition \ref{prop:Jkk}, and for all $1 \le k \le N$ such that $k \ne j$, we have $S_k \not\cong S_j$.

Form the following three subrings of $R$ (these subrings may be equal to $\{0\}$). Let $R_1:=S_i \oplus S_j \oplus J_{ij}$. Let $\mcI := \{1 \le k \le N : k \ne i, k \ne j, S_k \text{ is noncommutative}\}$ and let $R_2 := \bigoplus_{k \in \mcI} S_k$. Finally, let $\mcI' := \{1 \le k \le N : k \ne i, k \ne j, k \notin \mcI\}$ and let $R_3 := \bigoplus_{k \in \mcI'} (S_k \oplus J_{kk})$. Then, $R=R_1 \oplus R_2 \oplus R_3$, and the subrings satisfy the hypotheses of Proposition \ref{prop:R1R2R3}. Thus, $\sigma(R) = \min\{\sigma(R_1), \sigma(R_2), \sigma(R_3)\}$. But, $R$ is $\sigma$-elementary, so in fact $R=R_\ell$ for some $1 \le \ell \le 3$, and the other two subrings must be $\{0\}$. If $R=R_1$, then $R \cong A(n_i, q_i, q_j)$. If $R=R_2$, then $R$ is noncommutative and semisimple. Hence, $R \cong M_n(q)$ for some $n \ge 2$ by \cite[Proposition 5.4, Theorem 5.11]{PeruginelliWerner}. Finally, if $R=R_3$, then $R$ is commutative. We may then apply \cite[Theorem 4.8]{SwartzWerner} to conclude that either $R \cong \bigoplus_{i=1}^{\tau(q)} \F_q$, or $R \cong \F_q(+)\F_q^2$, depending on whether $R$ is semisimple or not.
\end{myproof}

\section{Rings without unity and covers by unital subrings}\label{sect:othermain}

Theorem \ref{thm:main} provides a complete classification of $\sigma$-elementary rings with unity, which in turn allows us to describe all the integers that occur as covering numbers of such rings. We now examine two related covering problems. First, what can be said of covering numbers for rings without a multiplicative identity? Second, how do our results change if we insist that each subring in a cover of a unital ring $R$ must contain $1_R$? As shown in Theorem \ref{thm:rngtoring}, these two problems are closely related, and an answer to the second question provides an answer to the first.

Recall the definitions of $\sigma_u$ and $\sigma_u$-elementary given in Definition \ref{def:sigmau}. It is clear that parts (3) and (4) of Lemma \ref{lem:basics} carry over to coverings of $R$ by unital subrings: for any two-sided ideal of $R$, we have $\sigma_u(R) \le \sigma_u(R/I)$; and when $R$ is coverable by unital subrings that are contained in maximal subrings, we may assume that each subring in a minimal cover is maximal. Going forward, we will make free and frequent use of these properties.

It is well known that any ring without unity can be embedded in a larger ring that contains a multiplicative identity. Using this construction, we will connect covers of rings without unity to covers of unital rings by unital subrings and prove Theorem \ref{thm:rngtoring}.

\begin{myproof}[of Theorem \ref{thm:rngtoring}]
We first proceed as in \cite[Section 3]{SwartzWerner} and reduce to the case where $R$ is finite and has prime power order. By results of Neumann \cite[Lemma 4.1,4.4]{Neumann} and Lewin \cite[Lemma 1]{Lewin}, $R$ contains a two-sided ideal $I$ of finite index such that $\sigma(R) = \sigma(R/I)$. Hence, we may assume that $R$ is finite. Next, the Chinese Remainder Theorem still applies to finite rings without identity, so $R \cong \bigoplus_{i=1}^t R_i$, where each $R_i$ is a ring and $|R_i| = p_i^{n_i}$ for some distinct primes $p_1, \ldots, p_t$. Lemma \ref{lem:AMM2.2} holds for rings without unity, so $\sigma(R) = \min_{1 \le i \le t}\{\sigma(R_i)\}$. Thus, we may assume that $|R|=p^n$ for some prime $p$ and some $n \ge 1$.

Now, as in \cite[Lemma 3.2, Proposition 3.3]{SwartzWerner}, one may prove that $pR$ is contained in every maximal subring of $R$. It follows that $\sigma(R) = \sigma(R/pR)$, and hence we may assume that $pR = \{0\}$. From here, we will embed $R$ in a unital ring $R'$ that has characteristic $p$.

Let $R' := \F_p \times R$ with multiplication given by the rule:
 \[ (n_1, r_1)(n_2, r_2) := (n_1n_2, n_1r_2 + n_2r_1 + r_1r_2).\]
Notice that, for any $(n, r) \in R'$, $(1,0)(n,r) = (n,r)(1,0) = (n,r)$, so $R'$ has a multiplicative identity.  We claim that $\sigma_u(R') = \sigma(R)$.
 
Indeed, let $\mcC$ be a minimal cover of $R$ by subrings.  Given $S \in \mcC$, define $\F_p \times S :=\{(n, s) \in R' : s \in S\}$, which is a unital subring of $R'$, and is proper in $R'$ because $S \subsetneq R$. Let $\mcC' := \{\F_p \times S : S \in \mcC\}$. If $(n, r) \in R'$, then $r \in S$ for some $S \in \mcC$, so $(n,r) \in \F_p \times S \in \mcC'$.  Hence, $\mcC'$ is a cover of $R'$ by unital subrings, and so $\sigma_u(R') \le |\mcC'| = |\mcC| = \sigma(R).$
 
Next, let $\mcU'$ be a minimal cover of $R'$ by unital subrings, which we now know must exist.  Let $\rho: R' \to R$ be the projection map defined by $\rho((n,r)) := r$.  The map $\rho$ is not multiplicative, but $\rho(S')$ is a proper subring of $R$ for any $S' \in \mcU'$. To see this, let $r_1, r_2 \in \rho(S')$. Then, $(n_1, r_1), (n_2, r_2) \in S'$ for some $n_1, n_2 \in \F_p$. Because $(1,0) \in S'$, both $(n_1, 0)$ and $(n_2, 0)$ are in $S'$. Thus, $S'$ contains both $(0,r_1)$ and $(0,r_2)$, and hence $r_1+r_2, r_1r_2 \in \rho(S')$. Moreover, $\rho(S') \ne R$, because if not, then $\{0\} \times R \subseteq S'$ and $(1,0) \in S'$, which implies that $S' = R'$.

Finally, given $r \in R$, there exists $S' \in \mcU'$ such that $(0,r) \in S'$. Hence, $r \in \phi(S')$ and $\{\phi(S') : S' \in \mcU'\}$ is a cover of $R$ of size at most $|\mcU'|$. Consequently, $\sigma(R) \le |\mcU'| = \sigma_u(R)$, completing the proof. 
\end{myproof}

In light of Theorem \ref{thm:rngtoring}, determining the covering number of any ring (with or without unity) reduces to the case of a unital ring $R$, for which we may consider either covers by subrings (which need not contain $1_R$) or covers by unital subrings. These problems can in turn be solved by classifying $\sigma$-elementary and $\sigma_u$-elementary rings. The classification of $\sigma$-elementary rings has been done in Theorem \ref{thm:main}. We now proceed to prove Theorem \ref{thm:mainunital}, which classifies $\sigma_u$-elementary rings. As we shall see, the class of $\sigma_u$-elementary rings is largely the same as the class of $\sigma$-elementary rings, with the greatest discrepancy occurring in the commutative semisimple case.

\begin{lem}\label{lem:covnumrelation}
Let $R$ be a ring with unity.  Then, $\sigma(R) \le \sigma_u(R)$, with equality if and only if the subrings in some minimal cover of $R$ each contain $1_R$.
\end{lem}
\begin{proof}
Note that every unital subring is a subring (but not conversely).
\end{proof}

\begin{lem}\label{lem:sigmaunotequalsigma}\mbox{}
\begin{enumerate}[(1)]
\item For any prime $p$, $\bigoplus_{i=1}^p \F_p$ is $\sigma$-elementary, but is not coverable by unital subrings.
\item For any prime $p$, $\bigoplus_{i=1}^{p+1} \F_p$ is $\sigma_u$-elementary, but is not $\sigma$-elementary. Furthermore, $\sigma_u\big(\bigoplus_{i=1}^{p+1} \F_p\big) = \sigma\big(\bigoplus_{i=1}^{p+1} \F_p\big) = p + \binom{p}{2}$.

\item The ring $A(1,2,2)$ is $\sigma_u$-elementary, but is not $\sigma$-elementary. Furthermore, $\sigma_u(A(1,2,2))= \sigma(A(1,2,2)) = 3$.
\end{enumerate}
\end{lem}
\begin{proof}
(1) Let $R = \bigoplus_{i=1}^p \F_p$. By \cite[Theorem 3.5]{Werner}, $R$ is coverable, but $\bigoplus_{i=1}^{p-1} \F_p$ is not, so $R$ is $\sigma$-elementary. As shown in \cite[Theorems 5.3]{Werner}, a minimal cover of $R$ requires every maximal subring of $R$. Since $\{0\}$ is the only maximal subring of $\F_p$, the Type $\Pi_1$ (see Definition \ref{def:pi1pi2}) maximal subrings of $R$ do not contain $1_R$. Thus, $R$ is not coverable by unital subrings.

(2) Let $R_i \cong \F_p$ for $1 \le i \le p+1$, and let $R = \bigoplus_{i=1}^{p+1} R_i$. Given $a = \sum_{i=1}^{p+1} a_i \in R$ with each $a_i \in R_i$, we must have $a_i = a_j$ for some $i \ne j$. Thus, $a$ lies in the Type $\Pi_2$ maximal subring of $R$ where $R_i$ is linked to $R_j$. There are $\binom{p+1}{2} = p + \binom{p}{2}$ such maximal subrings, and each one contains $1_R$. Thus, $\sigma_u(R) \le p + \binom{p}{2}$; but, $p + \binom{p}{2} = \sigma(R) \le \sigma_u(R)$ by Lemma \ref{lem:covnumrelation}, so $\sigma_u(R) = p + \binom{p}{2}$. By part (1), $R$ is $\sigma_u$-elementary. However, $R$ has $\bigoplus_{i=1}^p \F_p$ as a residue ring, and thus is not $\sigma$-elementary.

(3) Let $R = A(1,2,2)$, which is equal to the ring of upper triangular matrices over $\F_2$. The three unital subrings $S_1 := \left\{ \left(\begin{smallmatrix} a & 0 \\ 0 & b \end{smallmatrix} \right) : a, b \in \F_2\ \right\}$, $S_2 := \left\{ \left(\begin{smallmatrix} a & b \\ 0 & a \end{smallmatrix} \right) : a, b \in \F_2\ \right\}$, and $S_3 := \left\{ \left(\begin{smallmatrix} a & a+b \\ 0 & b \end{smallmatrix} \right) : a, b \in \F_2\ \right\}$ form a cover of $R$, so $\sigma_u(R) = 3$. The nonzero, proper residue rings of $R$ are isomorphic to $\F_2 \oplus \F_2$ and $\F_2$, neither of which has a unital cover. Hence, $R$ is $\sigma_u$-elementary. Finally, since $\sigma(\F_2 \oplus \F_2)=3$, $R$ is not $\sigma$-elementary.
\end{proof}

\begin{prop}\label{prop:whensigmaissigmau}
Let $R$ be a $\sigma$-elementary ring with unity. Then, $R$ is $\sigma_u$-elementary if and only if $R \not\cong \bigoplus_{i=1}^p \F_p$. Moreover, when $R$ is $\sigma_u$-elementary, we have $\sigma_u(R) = \sigma(R)$.
\end{prop}
\begin{proof}
We show that, with the exception of $\bigoplus_{i=1}^{p} \F_p$, all of the $\sigma$-elementary rings listed in Theorem \ref{thm:main} admit minimal covers by unital subrings. Indeed, if $q \ne p$, then every maximal subring of $\bigoplus_{i=1}^{\tau(q)} \F_q$ contains the identity of the ring \cite[Theorem 4.8]{Werner}. For any $q$, a cover of $\F_q (+) \F_q^{2}$ by unital maximal subrings is given in \cite[Example 6.1]{Werner}. For $n \ge 2$, maximal subrings of $M_n(q)$ were fully classified in \cite[Theorem 3.3]{PeruginelliWerner}, and all such subrings contain the identity matrix. Finally, for rings of AGL-type, minimal covers were constructed in \cite{SwartzWernerI}. As noted in \cite[Remarks 3.7, 4.3]{SwartzWernerI}, when $A(n, q_1, q_2)$ is $\sigma$-elementary, it admits a minimal cover by unital subrings.
\end{proof}

Proposition \ref{prop:whensigmaissigmau} proves the first half of Theorem \ref{thm:mainunital}. It remains to show that if $R$ is $\sigma_u$-elementary but not $\sigma$-elementary, then either $R \cong \bigoplus_{i = 1}^{p+1} \F_p$ or $R \cong A(1,2,2)$. Since we are now focusing on a ring with unity, we will employ many of the same assumptions and notations that were used in Sections \ref{sect:Peirce} through \ref{sect:summands}. In particular, by Lemma \ref{lem:FullReduction} it will suffice to consider rings of characteristic $p$ with 2-nilpotent radicals. For the remainder of this section, we will assume the following:

\begin{Not}\label{Not:sigmau} Let $R$ be a unital ring with characteristic $p$, and let  $J:=\msJ(R)$ be the Jacobson radical of $R$. Assume that $J^2 = \{0\}$. Let $\msS(R)$ be the set of all the semisimple complements to $J$ in $R$. Fix $S \in \msS(R)$, so that $R = S \oplus J$, and write $S = \bigoplus_{i=1}^N S_i$ for some simple rings $S_i$, each with unity $e_i$. For all $1 \le i, j \le N$, define $J_{ij} := e_i J e_j$ to be the $(S_i, S_j)$-bimodule obtained via a Peirce decomposition of $J$. Note that $J = \bigoplus_{i,j} J_{ij}$. Let $\lambda_{ij}$ be the length of $J_{ij}$ as an $(S_i, S_j)$-bimodule; if $J_{ij}=\{0\}$, then we take $\lambda_{ij}=0$.
\end{Not}

As in Section \ref{sect:J}, our strategy is to show that in a $\sigma_u$-elementary ring, most of the bimodules $J_{ij}$ must be zero. Much of the work done earlier in this paper still holds when stated only for unital subrings or $\sigma_u$-elementary rings. As noted earlier, Theorem \ref{thm:Jij=0}, and Lemma \ref{lem:Jkk} hold mutatis mutandis for unital maximal subrings and $\sigma$-elementary rings (see Remarks \ref{rem:mutmut1} and \ref{rem:mutmut2}).  In particular, we have the following structural results on a $\sigma_u$-elementary ring $R$.

\begin{lem}\label{lem:Jij-Jkk}
Let $R$ be as in Notation \ref{Not:sigmau}.
\begin{enumerate}[(1)]
\item If $R$ is $\sigma_u$-elementary, then there is at most one pair $(i,j)$ with $1 \le i < j \le N$ and $J_{ij} \ne \{0\}$. For this pair $(i,j)$, if $J_{ij} \ne \{0\}$, then $\lambda_{ij}=1$.

\item For all $1 \le k \le N$, if $S_k$ is a field and $\lambda_{kk}=1$, 
then $R$ is not $\sigma_u$-elementary.
\end{enumerate}
\end{lem}

\begin{lem}\label{lem:R/I}
Let $R$ be a $\sigma_u$-elementary ring that is not $\sigma$-elementary.
\begin{enumerate}[(1)]
\item There exists a two-sided ideal $I$ of $R$ such that $J \subseteq I$ and $R/I \cong \bigoplus_{i = 1}^p \F_p$.

\item $\sigma(R) = p + \binom{p}{2}$.

\item If $R$ has a residue ring isomorphic to $\bigoplus_{i = 1}^{p+1} \F_p$, then $R = \bigoplus_{i = 1}^{p+1} \F_p$ and $\sigma_u(R) = p + \binom{p}{2}$.
\end{enumerate}
\end{lem}
\begin{proof}
(1) Since $R$ is not $\sigma$-elementary, it projects nontrivially onto a $\sigma$-elementary ring $R/I$ such that $\sigma(R) = \sigma(R/I)$. By Proposition \ref{prop:whensigmaissigmau}, all $\sigma$-elementary rings are $\sigma_u$-elementary except for $\bigoplus_{i = 1}^p \F_p$.  If $R/I \not\cong \bigoplus_{i = 1}^p \F_p$, then $\sigma_u(R/I) = \sigma(R/I) = \sigma(R) \le \sigma_u(R)$, a contradiction to $R$ being $\sigma_u$-elementary.  Thus, there exists an ideal $I$ of $R$ such that $R/I \cong \bigoplus_{i = 1}^p \F_p$. 
Since $R/I$ is semisimple, $J \subseteq I$.

(2) With $I$ as in part (1), we have $\sigma(R) = \sigma(R/I) = p + \binom{p}{2}$.

(3) Assume that $R$ has a residue ring isomorphic to $\bigoplus_{i = 1}^{p+1} \F_p$. By Lemma \ref{lem:sigmaunotequalsigma}, $\bigoplus_{i = 1}^{p+1} \F_p$ is $\sigma_u$-elementary and has unital covering number $p+\binom{p}{2}$. By Lemma \ref{lem:covnumrelation} and part (2), $p+\binom{p}{2} = \sigma(R) \le \sigma_u(R).$ Since $R$ is $\sigma_u$-elementary, we must have $R = \bigoplus_{i = 1}^{p+1} \F_p$.
\end{proof}

Using Lemma \ref{lem:R/I}, we can impose another restriction on $R$ and assume that $R$ has a residue ring isomorphic to $\bigoplus_{i = 1}^p \F_p$, but does not have a residue ring isomorphic to $\bigoplus_{i = 1}^{p+1} \F_p$. Thus, we may assume that $S$ has the form
\begin{equation}\label{eq:Sform}
S = \Big(\bigoplus_{i=1}^p S_i\Big) \oplus \Big(\bigoplus_{i=p+1}^N S_i\Big),
\end{equation}
where $S_i \cong \F_p$ for all $1 \le i \le p$, and $S_i \not\cong \F_p$ for all $p+1 \le i \le N$.

\begin{lem}\label{lem:unitalPi1}
Let $S$ be as in \eqref{eq:Sform} and let $M$ be a maximal unital subring of $S$.
\begin{enumerate}[(1)]
\item If $M$ is Type $\Pi_1$, then $M$ contains $\bigoplus_{i=1}^p S_i$.
\item If $M$ does not contain $\bigoplus_{i=1}^p S_i$, then $M$ is Type $\Pi_2$ with $S_i$ linked to $S_j$ for some $1 \le i < j \le p$.
\end{enumerate}
\end{lem}
\begin{proof}
(1) Assume that $M$ is Type $\Pi_1$. Then, for some $1 \le i \le N$,
\begin{equation*}
M = S_1 \oplus \cdots \oplus S_{i-1} \oplus M_i \oplus S_{i+1} \oplus \cdots \oplus S_N,
\end{equation*}
where $M_i$ is a maximal subring of $S_i$. Suppose that $1 \le i \le p$. Then, $M_i = \{0\}$ because $S_i \cong \F_p$. But then, $1_R \notin M$, a contradiction. So, $i \ge p+1$, and $\bigoplus_{i=1}^p S_i \subseteq M$.

(2) If $\bigoplus_{i=1}^p S_i$ is not a subring of $M$, then by part (1) $M$ must be Type $\Pi_2$. In this case, the only way that $M$ will fail to contain $\bigoplus_{i=1}^p S_i$ is if $S_i$ is linked to $S_j$ for some $1 \le i < j \le p$.
\end{proof}

\begin{prop}\label{prop:lastbigprop}
Let $R$ be as Notation \ref{Not:sigmau}, and let $S$ be as in \eqref{eq:Sform}. Assume that $R$ is $\sigma_u$-elementary, but not $\sigma$-elementary.
\begin{enumerate}[(1)]
\item $\lambda_{ii} = 0$ for all $1 \le i \le p$.

\item $J_{ij} = \{0\}$ for all $1 \le i \le p$ and all $p+1 \le j \le N$.
\end{enumerate}
\end{prop}
\begin{proof}
(1) Suppose first that $\lambda_{ii} \ge 2$ for some $1 \le i \le p$. Then, $S_i \oplus J_{ii} \cong \F_p (+) \F_p^{\lambda_{ii}}$ (the idealization of $\F_p$ with the $\lambda_{ii}$-dimensional vector space $\F_p^{\lambda_{ii}}$), which projects onto $\F_p (+) \F_p^2$. By Lemma \ref{lem:Peirce}(4), $R$ has a residue ring isomorphic to $\F_p (+) \F_p^2$. Using this and Lemmas \ref{lem:covnumrelation} and \ref{lem:R/I}, we get
\begin{equation*}
p+\textstyle\binom{p}{2} \le \sigma_u(R) \le \sigma_u(\F_p (+) \F_p^2) = p+1.
\end{equation*}
Since $R$ is $\sigma_u$-elementary, this forces $R \cong \F_p (+) \F_p^2$. But, this contradicts the fact that $R$ is not $\sigma$-elementary.

So, we must have $\lambda_{ii} \le 1$ for all $1 \le i \le p$. However, by Lemma \ref{lem:Jij-Jkk}(2), if $\lambda_{ii}=1$ for any $1 \le i \le p$, then $R$ is not $\sigma_u$-elementary. We conclude that $\lambda_{ii}=0$ for all $1 \le i \le p$.

(2) Suppose by way of contradiction that $J_{k \ell} \ne \{0\}$ for some $1 \le k \le p$ and $p+1 \le \ell \le N$. By Lemma \ref{lem:Jij-Jkk}(1), there is at most one pair $(i,j)$ with $1 \le i < j \le N$ and $J_{ij} \ne \{0\}$. Thus, $(k, \ell)$ is unique, and $J_{k \ell}$ is a simple $(S_k,S_\ell)$-bimodule.

Assume without loss of generality that $k=p$. Let $T = \bigoplus_{i=1}^{p-1} S_i \cong \bigoplus_{i=1}^{p-1} \F_p$. We claim that $T$ is a two-sided ideal of $R$. To see this, note that by part (1), $J_{ii} = \{0\}$ for all $1 \le i \le p-1$. Hence, $J = J_{p \ell} \oplus \big(\bigoplus_{i=p+1}^N J_{ii}\big)$. Let $\whT = \big(\bigoplus_{i=p}^N S_i\big) \oplus J$. Then, $R = T \oplus \whT$, and elements from $T$ and $\whT$ mutually annihilate one another. Thus, $T$ is an ideal of $R$. Moreover, the action of $1+J$ on $T$ is trivial, so $T$ is contained in $S'$ for all $S' \in \msS(R)$. We will show that $\sigma_u(R) = \sigma_u(R/T)$, which contradicts the fact that $R$ is $\sigma_u$-elementary.

Let $\mcC_1$ be the set of all unital maximal subrings of $R$ that contain $T$, and let $\mcC_2$ be the set of all unital maximal subrings that do not. Let $M \in \mcC_2$, and suppose that $J \not\subseteq M$. Then, by \cite[Theorem 3.10]{SwartzWerner}, $M = S' \oplus I$ for some $S' \in \msS(R)$ and some maximal subideal $I \subseteq J$. But then, $T \subseteq S' \subseteq M$, a contradiction. So, each subring in $\mcC_2$ contains $J$. Use a bar to denote passage to $R/J \cong S$. By Lemma \ref{lem:unitalPi1},  $\olM$ is Type $\Pi_2$ with $\olS_i$ linked to $\olS_j$ for some $1 \le i < j \le p$, and this is true for every $M \in \mcC_2$.

From here, we will show that for each $a \in \whT \cong R/T$, there exists $\alpha \in T$ such that $a+\alpha$ is in a maximal subring containing $T$. Let $a \in \whT$, and write $a$ as $a = \alpha_p + b$, where $\alpha_p \in S_p$ and $b \in \big(\bigoplus_{i=p+1}^N S_i\big) \oplus J$. Note that $\alpha_p$ is equal to an element of $\F_p$. Let $\alpha_1, \ldots, \alpha_{p-1}$ be all of the elements of $\F_p$ that are not equal to $\alpha_p$. Considering $\alpha_i$ as an element of $S_i$ for each $1 \le i \le p-1$, we take $\alpha := \sum_{i=1}^{p-1} \alpha_i \in T$. 

Certainly, $a+\alpha$ lies in some unital maximal subring $M$ of $R$. Suppose that $M \in \mcC_2$. As noted above, $\olM$ is Type $\Pi_2$ with $\olS_i$ linked to $\olS_j$ for some $1 \le i < j \le p$. Since the only automorphism of $\F_p$ is the identity, this means that $\overline{\alpha_i} = \overline{\alpha_j}$. However, this is impossible, because $\alpha_1, \ldots, \alpha_p$ are all distinct elements of $\F_p$. Thus, $a+\alpha$ lies in some subring in $\mcC_1$.

To complete the proof, let $\mcC$ be a minimal cover of $R$ by unital subrings, and let $\mcA = \mcC \cap \mcC_1$. Then, each subring in $\mcA$ contains $T$, and their images in $R/T$ cover $R/T \cong \whT$. Hence, $\sigma_u(R/T) \le |\mcA| \le \sigma_u(R)$, which contradicts the fact that $R$ is $\sigma_u$-elementary.
\end{proof}

\begin{prop}\label{prop:R=R1}
Let $R$ be as described prior to Lemma \ref{lem:R/I}, and let $S$ be as in \eqref{eq:Sform}. Assume that $R$ is $\sigma_u$-elementary, but not $\sigma$-elementary. Then,
\begin{equation*}
R = \Big(\bigoplus_{i=1}^p S_i \Big) \oplus \Big(\bigoplus_{1 \le i, j \le p} J_{ij}\Big).
\end{equation*}
\end{prop}
\begin{proof}
By Proposition \ref{prop:lastbigprop}, we know that $J_{ij} = \{0\}$ for all $1 \le i \le p$ and $p+1 \le j \le N$. Decompose $R$ as $R = R_1 \oplus R_2$, where
\begin{align*}
R_1 &= \Big(\bigoplus_{i=1}^p S_i \Big) \oplus \Big(\bigoplus_{1 \le i, j \le p} J_{ij}\Big), \text{ and }\\
R_2 &= \Big(\bigoplus_{i=p+1}^N S_i \Big) \oplus \Big(\bigoplus_{p+1 \le i, j \le N} J_{ij}\Big).
\end{align*}

We claim that all unital maximal subrings of $R$ respect this decomposition. That is, if $M$ is a unital maximal subring of $R$, then either $M=M_1 \oplus R_2$ with $M_1$ maximal in $R_1$, or $M=R_1 \oplus M_2$ with $M_2$ maximal in $R_2$. When $J \not\subseteq M$, this can be shown just as in the proof of Proposition \ref{prop:R1R2R3}. When $J \subseteq M$, $M = T \oplus J$ for some maximal subring $T$ of $S$. In this case, the only way that $M$ would not respect the direct sum decomposition of $R$ is if $T$ were Type $\Pi_2$ with $S_i$ linked to $S_j$ for some $1 \le i \le p$ and $p+1 \le j \le N$. However, this is impossible, because $S_j \not\cong \F_p$ for all $p+1 \le j \le N$. 



Since every maximal subring of $R$ has the form $M_1 \oplus R_2$ or $R_1 \oplus M_2$, we can apply Lemma \ref{lem:AMM2.2}, which yields $\sigma_u(R) = \min\{\sigma_u(R_1), \sigma_u(R_2)\}$. Since $R$ is $\sigma_u$-elementary, this means that either $R = R_1$ or $R=R_2$. Because only $R_1$ has $\bigoplus_{i=1}^p \F_p$ as a residue ring, we conclude that $R=R_1$.
\end{proof}

\begin{prop}\label{prop:R=A(1,2,2)}
Let $R$ be a $\sigma_u$-elementary ring that is not $\sigma$-elementary, and such that $R \not\cong \bigoplus_{i=1}^{p+1} \F_p$. Then, $R \cong A(1,2,2)$.
\end{prop}
\begin{proof}
By Lemma \ref{lem:R/I}, $R$ has no residue ring isomorphic to $\bigoplus_{i=1}^{p+1} \F_p$. By Proposition \ref{prop:R=R1}, $R = \big(\bigoplus_{i=1}^p \F_p \big) \oplus J$, and $J \ne \{0\}$ because $R$ is coverable by unital subrings. By Proposition \ref{prop:lastbigprop}, $J_{ii} = \{0\}$ for all $1 \le i \le p$. Moreover, by Lemma \ref{lem:Jij-Jkk}(1), there is a unique pair $(i,j)$ with $i \ne j$ such that $J_{ij} \ne \{0\}$, and $\lambda_{ij}=1$. 

Thus, $S_i \oplus S_j \oplus J_{ij} \cong A(1,p,p)$, which has unital covering number $p+1$. By \ref{lem:Peirce}, $A(1,p,p)$ occurs as a residue ring of $R$. Thus, $p+\binom{p}{2} \le \sigma_u(R) \le p+1$, which forces $p=2$ and $\sigma_u(R)=3$. Since $R$ is $\sigma_u$-elementary, we must have $R = S_i \oplus S_j \oplus J_{ij} \cong A(1,2,2)$.
\end{proof}

\begin{myproof}[of Theorem \ref{thm:mainunital}]
Apply Proposition \ref{prop:whensigmaissigmau}, Lemma \ref{lem:R/I}, and Proposition \ref{prop:R=A(1,2,2)}. The covering numbers for $\bigoplus_{i=1}^{p+1} \F_p$ and $A(1,2,2)$ were found in Lemma \ref{lem:sigmaunotequalsigma}.
\end{myproof}

\section{Bounds on the number of integers that are covering numbers of rings}
\label{sect:E(N)bounds}

This section is dedicated to the proof of Corollary \ref{cor:density}.    Recall that \[\mathscr{E}(N) := \{m : m \le N, \sigma(R) = m \text{ for some ring } R\}.\] 
Let $\log x$ denote the binary (base 2) logarithm. Our goal is to show that $|\mathscr{E}(N)|$ is bounded above by $cN/\log N$ for some positive constant $c$. Rather than attempting to find an optimal value for $c$, we will be content with a value that is easy to verify.

Let $\pi$ be the prime counting function, which counts the number of primes less than or equal to a given positive number.  The following bounds on $\pi(x)$ will be useful in later estimates.

\begin{lem}\label{lem:piupperbound}
For all $x >1 $, $\pi(x) < 2x/\log x$. For all $x > 5$, $x/\log x < \pi(x)$.
\end{lem}
\begin{proof}
Let $\ln x$ denote the natural logarithm. By \cite[Corollary 1]{RosserSchoenfeld}, we have $\pi(x) < (1.25506 x)/(\ln x)$ when $x > 1$, and $\pi(x) > x/\ln(x)$ when $x \ge 17$.  The remaining cases for the lower bound follow from inspection.
\end{proof}

Let $\Pi$ be the prime power counting function, i.e.,
\[\Pi(x) :=  |\{m \in \N : 2 \le m \le x \text{ and } m \text{ is a prime power} \}|.\]
Using Lemma \ref{lem:piupperbound}, it is not difficult to get an upper bound on $\Pi(x)$.

\begin{lem}
 \label{lem:primepowerbound}
 Let $x > 1$.  Then, $\Pi(x) < 8x/\log x$.
\end{lem}

\begin{proof}
Let $\ell := \lfloor \log x \rfloor$. Then,
\begin{align*}
\Pi(x) &= \sum_{r=1}^{\ell} \pi(x^{1/r}) = \pi(x) + \sum_{r=2}^{\ell} \pi(x^{1/r}) \le \pi(x) + (\ell-1)\pi(x^{1/2})\\
       &\le \dfrac{2x}{\log x} + (\ell-1) \dfrac{2x^{1/2}}{(1/2)\log(x)} \le \dfrac{2x}{\log x} + 4x^{1/2} < \dfrac{8x}{\log x}.
\end{align*}
\end{proof}

Using Lemma \ref{lem:primepowerbound}, we can provide bounds on the number of integers that are expressible in terms of the formulas given in the statement of Theorem \ref{thm:main}.

\begin{lem}\label{lem:easybounds}
Let $N > 1$ be a natural number.  
\begin{enumerate}[(1)]
\item The number of integers $m$, $2 \le m \le N$, that can be expressed as $\tau(q) \nu(q) + d \binom{\tau(q)}{2}$ for some prime power $q = p^d$ is less than $8N/\log N$.
\item The number of integers $m$, $2 \le m \le N$, that can be expressed as $q + 1$ for some prime power $q = p^d$ is less than $8 N/\log N$.
\item The number of integers $m$, $2 \le m \le N$, that can be expressed as \[\frac{1}{a} \prod_{k=1,\\ a \nmid k}^{n-1} (q^n - q^k) + \sum_{k=1,\\ a \nmid k}^{\lfloor n/2 \rfloor} \binom{n}{k}_q,\]
where $q$ is some prime power, $n \ge 2$, $a$ is the smallest prime divisor of $n$, and $\binom{n}{k}_q$ is the $q$-binomial coefficient is less than $56 N/\log N$.
\item Let $N > 1$ be a natural number.  The number of integers $m$, $2 \le m \le N$, that can be expressed as \[q^n + \binom{n}{d}_{q_1} + \omega(d),\]
where $q$ and $q_1$ are prime powers, $q = q_1^d$, $n \ge 3$, and $d < n$ is less than $72 N/\log N$.
\end{enumerate}
\end{lem}
\begin{proof}
The first two parts follow from the fact that there is one such integer $m$ of the desired form for each prime power. Hence, the number of such integers $m$ is bounded above by $\Pi(N)$, which is less than $8N/\log N$ by Lemma \ref{lem:primepowerbound}.

For (3), let 
\begin{equation*}
P:= \frac{1}{a} \prod_{k=1,\\ a \nmid k}^{n-1} (q^n - q^k) \; \text{ and } \; S:= \sum_{k=1,\\ a \nmid k}^{\lfloor n/2 \rfloor} \binom{n}{k}_q.
\end{equation*}
We first note that $N \ge P+S \ge \tfrac{1}{2}q^2$, from which we conclude that $q \le \sqrt{2N}$, and there at most $\Pi(\sqrt{2N})$ choices for $q$.

Next, by \cite[Lemma 3.31]{SwartzWernerI} we have $P \ge q^{n(n-(n/a)-1)}$, so
\begin{equation*}
N \ge P+S > q^{n(n-(n/a)-1)} \ge q^{n(n-2)/2} \ge q^{(n-2)^2/2}.
\end{equation*}
It follows that $(n-2)^2 < 2 \log N$ and $n < 2 + \sqrt{2\log N}$. Since there is exactly one integer expressible as $P+S$ for each pair $(q,n)$, the number of such integers at most $N$ having this form is bounded above by
\begin{equation*}
\Pi(\sqrt{2N}) \cdot (2 + \sqrt{2\log N}) < \frac{8 \sqrt{2N}(2 + \sqrt{2\log N})}{(1/2)\log(2N)} < \frac{56N}{\log N}.
\end{equation*}

Finally, for (4), since $q^3 \le q^n < N$,  there are at most $\Pi(\sqrt[3]{N})$ choices for $q$, and, since $2^n \le q^n < N$ and $d < n$, there are at most $\log N$ choices for $n$ for a fixed $q$, and at most $\log N$ choices for $d$ given a fixed $n$.  Hence, the number of integers expressible in the desired form is bounded above by the number of triples $(q,n,d)$, which is bounded above by
 \[\Pi(\sqrt[3]{N}) (\log N)^2 < \frac{8 \sqrt[3]{N} (\log N)^2}{(1/3)\log N} < \frac{72 N}{\log N}.\]
 \end{proof}

We now have what we need to prove Corollary \ref{cor:density}.

\begin{proof}[Proof of Corollary \ref{cor:density}]
The upper bound of $144N/\log N$ follows from Lemma \ref{lem:easybounds}.  For the lower bound, we note that by Theorem \ref{thm:main} (2), every integer of the form $q + 1$, where $q$ is a prime power, is a covering number of a ring with unity. Thus, $|\mathscr{E}(N)| \ge \Pi(N) - 1$, and for $N \ge 5$ we have $\Pi(N)-1 \ge \pi(N) > N/\log N$.
\end{proof}

\bibliographystyle{plain}
\bibliography{CoveringNumbersOfRings}

\begin{thebibliography}{10}

\bibitem{Bhargava}
Mira Bhargava.
\newblock Groups as unions of proper subgroups.
\newblock {\em Amer. Math. Monthly}, 116(5):413--422, 2009.

\bibitem{Britnell1}
J.~R. Britnell, A.~Evseev, R.~M. Guralnick, P.~E. Holmes, and A.~Mar\'{o}ti.
\newblock Sets of elements that pairwise generate a linear group.
\newblock {\em J. Combin. Theory Ser. A}, 115(3):442--465, 2008.

\bibitem{Britnell2}
J.~R. Britnell, A.~Evseev, R.~M. Guralnick, P.~E. Holmes, and A.~Mar\'{o}ti.
\newblock Corrigendum to ``{S}ets of elements that pairwise generate a linear
  group'' [{J}. {C}ombin. {T}heory {S}er. {A} 115 (3) (2008) 442--465].
\newblock {\em J. Combin. Theory Ser. A}, 118(3):1152--1153, 2011.

\bibitem{BryceFedriSerena}
R.~A. Bryce, V.~Fedri, and L.~Serena.
\newblock Subgroup coverings of some linear groups.
\newblock {\em Bull. Austral. Math. Soc.}, 60(2):227--238, 1999.

\bibitem{CaiWerner}
Merrick Cai and Nicholas~J. Werner.
\newblock Covering numbers of upper triangular matrix rings over finite fields.
\newblock {\em Involve}, 12(6):1005--1013, 2019.

\bibitem{Clark}
Pete~L. Clark.
\newblock Covering numbers in linear algebra.
\newblock {\em Amer. Math. Monthly}, 119(1):65--67, 2012.

\bibitem{Cohen}
Jonathan Cohen.
\newblock On rings as unions of four subrings.
\newblock \url{https://arxiv.org/abs/2008.03803}, 2020.

\bibitem{Cohn}
J.~H.~E. Cohn.
\newblock On {$n$}-sum groups.
\newblock {\em Math. Scand.}, 75(1):44--58, 1994.

\bibitem{Crestani}
Eleonora Crestani.
\newblock Sets of elements that pairwise generate a matrix ring.
\newblock {\em Comm. Algebra}, 40(4):1570--1575, 2012.

\bibitem{DetomiLucchini}
Eloisa Detomi and Andrea Lucchini.
\newblock On the structure of primitive {$n$}-sum groups.
\newblock {\em Cubo}, 10(3):195--210, 2008.

\bibitem{DonovenKappe}
Casey Donoven and Luise-Charlotte Kappe.
\newblock Finite coverings of semigroups and related structures.
\newblock \url{https://arxiv.org/pdf/2002.04072.pdf}, 2020.

\bibitem{GagolaKappe}
Stephen~M. Gagola, III and Luise-Charlotte Kappe.
\newblock On the covering number of loops.
\newblock {\em Expo. Math.}, 34(4):436--447, 2016.

\bibitem{Gallian}
{Joseph A.} Gallian.
\newblock {\em Contemporary abstract algebra}.
\newblock Cengage Learning, {N}inth edition, 2017.

\bibitem{Garonzi}
Martino Garonzi.
\newblock Finite groups that are the union of at most 25 proper subgroups.
\newblock {\em J. Algebra Appl.}, 12(4):1350002, 11, 2013.

\bibitem{GaronziKappeSwartz}
Martino Garonzi, Luise-Charlotte Kappe, and Eric Swartz.
\newblock On integers that are covering numbers of groups.
\newblock \textit{Experimental Mathematics},
  \url{https://doi.org/10.1080/10586458.2019.1636425}, 2019.

\bibitem{Ghosh}
Soham Ghosh.
\newblock The {Z}ariski covering number for vector spaces and modules.
\newblock {\em Comm. Algebra}, 50(5):1994--2017, 2022.

\bibitem{Holmes}
P.~E. Holmes.
\newblock Subgroup coverings of some sporadic groups.
\newblock {\em J. Combin. Theory Ser. A}, 113(6):1204--1213, 2006.

\bibitem{Kappe}
Luise-Charlotte Kappe.
\newblock Finite coverings: a journey through groups, loops, rings and
  semigroups.
\newblock In {\em Group theory, combinatorics, and computing}, volume 611 of
  {\em Contemp. Math.}, pages 79--88. Amer. Math. Soc., Providence, RI, 2014.

\bibitem{Khare}
Apoorva Khare.
\newblock Vector spaces as unions of proper subspaces.
\newblock {\em Linear Algebra Appl.}, 431(9):1681--1686, 2009.

\bibitem{KhareTikaradze}
Apoorva Khare and Akaki Tikaradze.
\newblock Covering modules by proper submodules.
\newblock {\em Comm. Algebra}, 50(2):498--507, 2022.

\bibitem{Lang}
Serge Lang.
\newblock {\em Algebra}, volume 211 of {\em Graduate Texts in Mathematics}.
\newblock Springer-Verlag, New York, third edition, 2002.

\bibitem{Lewin}
Jacques Lewin.
\newblock Subrings of finite index in finitely generated rings.
\newblock {\em J. Algebra}, 5:84--88, 1967.

\bibitem{LucchiniMaroti}
Andrea Lucchini and Attila Mar\'{o}ti.
\newblock Rings as the unions of proper subrings.
\newblock {\em Algebr. Represent. Theory}, 15(6):1035--1047, 2012.

\bibitem{Maroti}
Attila Mar\'{o}ti.
\newblock Covering the symmetric groups with proper subgroups.
\newblock {\em J. Combin. Theory Ser. A}, 110(1):97--111, 2005.

\bibitem{McDonald}
Bernard~R. McDonald.
\newblock {\em Finite rings with identity}.
\newblock Pure and Applied Mathematics, Vol. 28. Marcel Dekker, Inc., New York,
  1974.

\bibitem{Neumann}
B.~H. Neumann.
\newblock Groups covered by permutable subsets.
\newblock {\em J. London Math. Soc.}, 29:236--248, 1954.

\bibitem{PeruginelliWerner}
G.~Peruginelli and N.~J. Werner.
\newblock Maximal subrings and covering numbers of finite semisimple rings.
\newblock {\em Comm. Algebra}, 46(11):4724--4738, 2018.

\bibitem{Pierce}
Richard~S. Pierce.
\newblock {\em Associative algebras}, volume~9 of {\em Studies in the History
  of Modern Science}.
\newblock Springer-Verlag, New York-Berlin, 1982.

\bibitem{RosserSchoenfeld}
J.~Barkley Rosser and Lowell Schoenfeld.
\newblock Approximate formulas for some functions of prime numbers.
\newblock {\em Illinois J. Math.}, 6:64--94, 1962.

\bibitem{Scorza}
Gaetano Scorza.
\newblock I gruppi che possone pensarsi come somma di tre lori sottogruppi.
\newblock {\em Boll. Un. Mat. Ital.}, 5:216--218, 1926.

\bibitem{Swartz}
Eric Swartz.
\newblock On the covering number of symmetric groups having degree divisible by
  six.
\newblock {\em Discrete Math.}, 339(11):2593--2604, 2016.

\bibitem{SwartzWerner}
Eric Swartz and Nicholas~J. Werner.
\newblock Covering numbers of commutative rings.
\newblock {\em J. Pure Appl. Algebra}, 225(8):106622, 17, 2021.

\bibitem{SwartzWernerI}
Eric Swartz and Nicholas~J. Werner.
\newblock {A} new infinite family of $\sigma$-elementary rings.
\newblock \url{https://arxiv.org/abs/2211.10313}, 2022.

\bibitem{Tomkinson}
M.~J. Tomkinson.
\newblock Groups as the union of proper subgroups.
\newblock {\em Math. Scand.}, 81(2):191--198, 1997.

\bibitem{Werner}
Nicholas~J. Werner.
\newblock Covering numbers of finite rings.
\newblock {\em Amer. Math. Monthly}, 122(6):552--566, 2015.

\end{thebibliography}

\end{document}